\documentclass[journal,twoside,web]{ieeecolor}

\usepackage{generic}
\usepackage{cite}
\usepackage{amsmath,amssymb,amsfonts,amsthm,balance,bm,mathtools}
\usepackage{algorithmic}
\usepackage{graphicx}
\usepackage{algorithm,algorithmic}
\usepackage{textcomp}
\usepackage{subfigure}
\usepackage{hyperref}

\newtheorem{theorem}{Theorem}
\newtheorem{lemma}{Lemma}

\newtheorem{definition}{Definition}

\newtheorem{remark}{Remark}

\newcommand{\differential}{{\rm{d}}}
\newcommand{\tr}{{\rm{trace}}}
\newcommand{\hess}{{\rm{Hess}}}
\newcommand{\dom}{{\rm{dom}}}
\newcommand{\gen}{{\rm{L}}}
\newcommand{\RNum}[1]{\uppercase\expandafter{\romannumeral #1\relax}}
\renewcommand{\qedsymbol}{\hfill\ensuremath{\blacksquare}}

\def\BibTeX{{\rm B\kern-.05em{\sc i\kern-.025em b}\kern-.08em
    T\kern-.1667em\lower.7ex\hbox{E}\kern-.125emX}}
\begin{document}
\title{Set Invariance with Probability One\\for Controlled Diffusion: Score-based Approach}
\author{Wenqing Wang, Alexis M.H. Teter, Murat Arcak, \IEEEmembership{Fellow, IEEE}, Abhishek Halder, \IEEEmembership{Senior Member, IEEE}
\thanks{Wenqing Wang (wqwang@iastate.edu) and Abhishek Halder (ahalder@iastate.edu) are with the Department of Aerospace Engineering, Iowa State University.}
\thanks{Alexis M.H. Teter (amteter@ucsc.edu) is with the Department of Applied Mathematics, University of California, Santa Cruz.}
\thanks{Murat Arcak (arcak@berkeley.edu) is with the Department of Electrical Engineering and Computer Sciences, University of California, Berkeley.}
\thanks{This research was supported in part by NSF award 2111688 and the ARCS Foundation Fellowship.}}

\maketitle

\begin{abstract}
Given a controlled diffusion and a connected, bounded, Lipschitz set, when is it possible to guarantee controlled set invariance with probability one? In this work, we answer this question by deriving the necessary and sufficient conditions for the same in terms of gradients of certain log-likelihoods---a.k.a. score vector fields---for two cases: given finite time horizon and infinite time horizon. The deduced conditions comprise a score-based test that provably certifies or falsifies the existence of Markovian controllers for given controlled set invariance problem data. Our results are constructive in the sense when the problem data passes the proposed test, we characterize all controllers guaranteeing the desired set invariance. When the problem data fails the proposed test, there does not exist a controller that can accomplish the desired set invariance with probability one. The computation in the proposed tests involve solving certain Dirichlet boundary value problems, and in the finite horizon case, can also account for additional constraint of hitting a target subset at the terminal time. We illustrate the results using several semi-analytical and numerical examples. 
\end{abstract}


\section{Introduction}\label{sec:introduction}
Consider a controlled It\^{o} diffusion process $\bm{x}_{t}^{\bm{u}}\in\mathbb{R}^{n}$ governed by
\begin{align}
\differential\bm{x}_{t}^{\bm{u}} = \left(\bm{f}\left(t,\bm{x}_{t}^{\bm{u}}\right) + \bm{G}\left(t,\bm{x}_{t}^{\bm{u}}\right)\bm{u}\right)\differential t + \bm{\sigma}\left(t,\bm{x}_{t}^{\bm{u}}\right)\differential \bm{w}_t
\label{ControlledSDE}   
\end{align}
over a time interval $\mathcal{I}$, where the control input $\bm{u}\in\mathbb{R}^{m}$, and $\bm{w}_t\in\mathbb{R}^{p}$ denotes the standard Wiener process. We think of $\bm{f},\bm{\sigma}$ in \eqref{ControlledSDE} modeling a noisy prior dynamics
\begin{align}
\differential\bm{x}_{t} = \bm{f}\left(t,\bm{x}_{t}\right)\differential t + \bm{\sigma}\left(t,\bm{x}_{t}\right)\differential \bm{w}_t,
\label{UnconditionedSDE}
\end{align}
and the input matrix $\bm{G}$ modeling the control authority. Let the diffusion tensor 
\begin{align}
\bm{\Sigma}(t,\bm{x}_t):=\bm{\sigma}\left(t,\bm{x}_t\right)\bm{\sigma}^{\top}\left(t,\bm{x}_t\right),
\label{defSigma}    
\end{align}
which is an $n\times n$ positive semidefinite matrix field.

We assume that the drift coefficient $\bm{f}$ and the diffusion coefficient $\bm{\sigma}$ satisfy the following regularity. 
\begin{itemize}
\item[\textbf{A1.}] (\textbf{Non-explosion and Lipschitz coefficients}) There exist constants $c_1,c_2>0$ such that $\forall\bm{x},\bm{y}\in\mathbb{R}^n$, $\forall t\in\mathcal{I}$, 
\begin{align*}
\|\bm{f}(t,\bm{x})\|_2 + \|\bm{\sigma}\left(t,\bm{x}\right)\|_{2} 
&\leq 
c_1\left(1 + \|\bm{x}\|_2\right), \\
\|\bm{f}(t,\bm{x}) - \bm{f}(t,\bm{y})\|_2 
&\leq 
c_2 \|\bm{x}-\bm{y}\|_2.    
\end{align*}
\item[\textbf{A2.}] (\textbf{Uniformly lower bounded diffusion tensor}) There exists a constant $c_3>0$ such that $\forall \bm{x}\in\mathbb{R}^{n}$, $\forall t\in\mathcal{I}$,
\begin{align*}
    \langle \bm{x}, \bm{\Sigma}(t,\bm{x})\bm{x}\rangle 
    \geq 
    c_3 \|\bm{x}\|_2^2.
\end{align*}  
\end{itemize}
Assumption \textbf{A1} ensures \cite[p. 66]{oksendal2013stochastic} the strong solution of \eqref{UnconditionedSDE}, i.e., the existence-uniqueness of the sample path $\bm{x}_t$ that satisfies \eqref{UnconditionedSDE} for some given initial condition. Assumptions \textbf{A1} and \textbf{A2} together ensure \cite[p. 41]{friedman2008partial} that the diffusion process $\bm{x}_t$ admits transition density $p(s,\bm{x},t,\bm{y})$ for all $0\leq s<t < \infty$ that is strictly positive and everywhere continuous. Assumption \textbf{A2} is sometimes referred to as the \emph{uniform ellipticity} condition \cite[p. 364]{karatzas2014brownian} and implies that the matrix field $\bm{\Sigma}$ is positive definite, or equivalently that $\bm{\sigma}^{\top}$ has trivial nullspace: $\mathcal{N}\left(\bm{\sigma^{\top}}\right)=\{\bm{0}\}$.

We make the following assumption on the set $\mathcal{X}\subset\mathbb{R}^{n}$ to be rendered invariant by the controlled diffusion \eqref{ControlledSDE}.

\noindent\textbf{A3.} The set $\mathcal{X}\subset\mathbb{R}^{n}$ is connected, bounded and Lipschitz\footnote{i.e., the connected set $\mathcal{X}$ is open, relatively compact and the boundary $\partial\mathcal{X}$ is the graph of a locally Lipschitz function.}.

We suppose that the initial state $\bm{x}_{0}\in\mathcal{X}$. We are interested in \emph{certifying or falsifying} whether it is possible to design Markovian control 
\begin{align}
\bm{u}:\mathcal{I}\times\mathcal{X}\mapsto\mathbb{R}^{m},
\label{MarkovianControl}
\end{align}
such that the controlled process $\bm{x}_{t}^{\bm{u}}$ stays in $\mathcal{X}$ \emph{with probability one} over the desired time horizon $\mathcal{I}$, i.e., achieves \emph{almost sure (a.s.) set invariance} over $\mathcal{I}$. 

We consider two types of time intervals: 
\begin{itemize}
    \item $\mathcal{I}= [0,T]$ is prescribed for a given $T<\infty$,

    \item $\mathcal{I}= [0,\infty)$.
\end{itemize}
Given the problem data $\left(\bm{f},\bm{G},\bm{\sigma},\mathcal{X},\mathcal{I}\right)$, we deduce a computational procedure or test that provably certifies or falsifies whether controlled set invariance with probability one is possible or not. The proposed test is necessary and sufficient for a.s. set invariance for controlled diffusion.

Our results are \emph{constructive}. When the problem data pass the proposed test, then we not only \emph{certify the existence} of controller $\bm{u}$ achieving the desired set invariance but also \emph{characterize all such controllers}. 

When the problem data fail the proposed test, then the existence of such controller $\bm{u}$ is \emph{falsified}. Falsification means that designing a controller that achieves a.s. set invariance for problem data tuple $\left(\bm{f},\bm{G},\bm{\sigma},\mathcal{X},\mathcal{I}\right)$, is impossible. Such impossibility, then, is not due to any algorithmic conservatism but is in fact a fundamental mathematical limitation imposed by the problem data. 

We will explain how the precise nature of computations required to perform the test differ for $\mathcal{I}=[0,T]$ and for $\mathcal{I}=[0,\infty)$. However, at a high level, computations for both these cases follow a common two step template:\\
\noindent\textbf{step 1.} \emph{compute a score vector field} $\bm{s}_{T}(t,\bm{x})$ when $\mathcal{I}=[0,T]$, and $\bm{s}_{\infty}(t,\bm{x})$ when $\mathcal{I}=[0,\infty)$. This computation uses problem data $\bm{f},\bm{\sigma},\mathcal{X},\mathcal{I}$ but not $\bm{G}$.\\

\vspace*{-0.1in}
\noindent\textbf{step 2.} \emph{characterize the controllers} $\bm{u}$, or the lack thereof, via the solutions of a static linear system of the form 
\begin{align}
\bm{G}\left(t,\bm{x}\right)\bm{u}\left(t,\bm{x}\right)=\begin{cases}
\bm{s}_{T}\left(t,\bm{x}\right) & \text{if}\quad \mathcal{I}=[0,T],\\
\bm{s}_{\infty}(t,\bm{x}) &\text{if}\quad \mathcal{I}=[0,\infty),
\end{cases}
\label{BackoutToControl}
\end{align}
for all $(t,\bm{x})\in\mathcal{I}\times\mathcal{X}$.

We will show that \textbf{step 1} involves solving suitable PDE boundary value problems (BVPs) involving the generator of the uncontrolled dynamics \eqref{UnconditionedSDE}. We will exemplify their solutions both analytically and numerically.

It turns out that the solution of \textbf{step 1} used in \textbf{step 2}, has some structure. In particular, for $\mathcal{I}=[0,T]$, the vector field $\bm{s}_{T}$ depends on $(t,\bm{x})$ and is parameterized by the given $T<\infty$. The corresponding $\bm{u}$ in \eqref{BackoutToControl}, if exists, is a \emph{time-varying state feedback} even when $\bm{G},\bm{\sigma}$ have no explicit dependence on $t$.

In contrast, for $\mathcal{I}=[0,\infty)$, the vector field $\bm{s}_{\infty}$ depends on $(t,\bm{x})$ with some additional structure. This additional structure is such that the controller $\bm{u}$ guaranteeing a.s. infinite horizon set invariance must necessarily be \emph{time-invariant feedback} provided $\bm{G},\bm{\sigma}$ have no explicit dependence on $t$.

We will also see that the two step template can constructively certify or falsify \emph{additional performance specifications beyond set invariance}. For example, in prescribed time $T<\infty$ case, it could be of interest to hit a target or goal set $\mathcal{X}_{T}\subseteq\mathcal{X}$ at the terminal time $t=T$, in addition to rendering the set $\mathcal{X}$ invariant for all $t\in\mathcal{I}=\left[0,T\right]$; see Fig. \ref{fig:introfigure}. This can be handled in the proposed framework by computing the vector field $\bm{s}_{T}(t,\bm{x})$ in \textbf{step 1} with problem data $\left(\bm{f},\bm{\sigma},\mathcal{X},\mathcal{X}_{T},\mathcal{I}\right)$ while keeping \textbf{step 2} unchanged.


\subsection{Motivation and related works}\label{subsec:RelatedWorks}
\begin{figure}[t]
\centering
\includegraphics[width=0.8\linewidth]{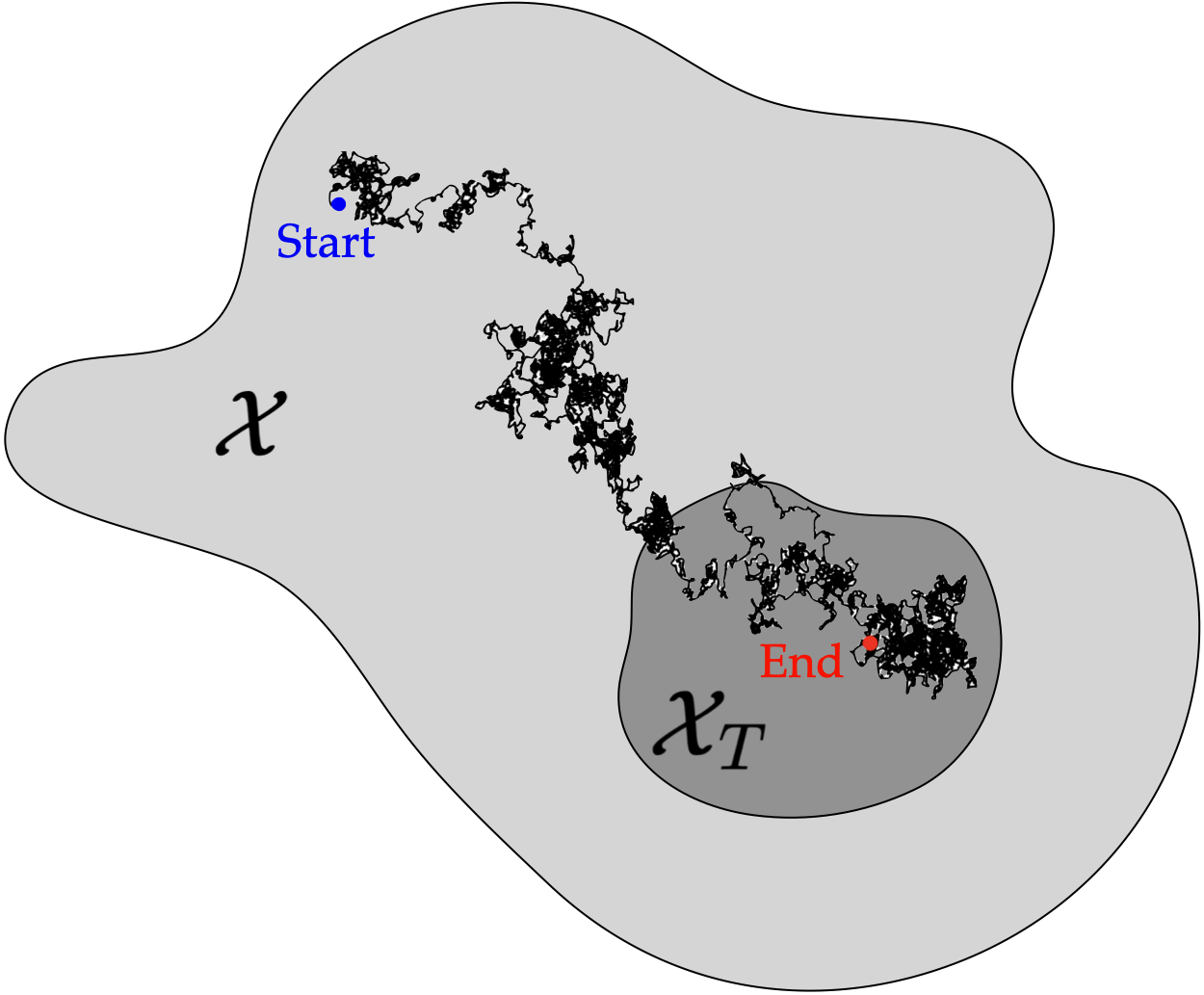}
\caption{Almost sure fixed horizon controlled set invariance with problem data $\left(\bm{f},\bm{\sigma},\mathcal{X},\mathcal{X}_{T},[0,T]\right)$. The problem is to design Markovian control \eqref{MarkovianControl} such that the controlled path governed by \eqref{ControlledSDE} satisfies $\bm{x}_{t}^{\bm{u}}\in\mathcal{X}$ for all $t\in[0,T]$ and $\bm{x}_{T}^{\bm{u}}\in\mathcal{X}_{T}\subseteq\mathcal{X}$ with probability one.}
\label{fig:introfigure}
\end{figure}

The notion of set invariance is motivated by the need to verify safety for controlled dynamics. For instance, in automotive applications \cite{ames2016control,son2019safety,black2023safety,xiao2023safe}, the set $\mathcal{X}$ in Fig. \ref{fig:introfigure} can signify a joint specification of safety bounds in the position (e.g., lateral deviation from the center of the lane, distance from the neighboring vehicle) and velocity (e.g., speed limit).  

The a.s. controlled set invariance, as considered here, is the strongest notion of probabilistic invariance in the stochastic context. Weaker notions such as invariance with high probability may suffice in practice, and have been studied before \cite{esfahani2016stochastic,vinod2021stochastic,keller2024mean,colonius2008near}. Nonetheless, a computational test that can certify or falsify a.s. controlled set invariance for controlled diffusions remains difficult. The intent of this work is to close this gap.

For diffusions, there is a rich literature on the related notion of \emph{viability theory} without control, see e.g., \cite[Ch. 1.5]{aubin2009viability}. For controlled diffusions, invariance conditions are available in terms of stochastic tangent cones \cite{aubin1998viability}, in terms of solutions of a second order Hamilton-Jacobi-Bellman (HJB) PDE \cite{bardi2002geometric, bardi2005almost,quincampoix2005stochastic}, and in terms of the distance-to-set function \cite{da2001stochastic,da2007stochastic}. It is also known \cite{da2004invariance,buckdahn2010another} that the (infinite horizon) a.s. invariance for the set $\mathcal{X}$ w.r.t. \eqref{ControlledSDE} is equivalent to its invariance w.r.t. the following deterministic control system\footnote{In \eqref{equivdetcontrolsystem}, $\bm{\sigma}_{j}$ denotes the $j$th column of the $n\times p$ matrix field $\bm{\sigma}$.} 
\begin{align}
\dot{\bm{x}}_{t}^{\bm{u,\bm{v}}} = &\bm{f}(t,\bm{x}_{t}^{\bm{u,\bm{v}}}) -\frac{1}{2}\sum_{j=1}^{p} \left(\nabla_{\bm{x}_{t}^{\bm{u,\bm{v}}}}\bm{\sigma}_{j}(t,\bm{x}_{t}^{\bm{u,\bm{v}}})\right)\bm{\sigma}_{j}(t,\bm{x}_{t}^{\bm{u,\bm{v}}}) \nonumber\\
&+ \bm{G}(t,\bm{x}_{t}^{\bm{u,\bm{v}}})\bm{u} + \bm{\sigma}(t,\bm{x}_{t}^{\bm{u,\bm{v}}})\bm{v},
\label{equivdetcontrolsystem}    
\end{align}
with additional controls $\bm{v}\in L^{1}_{\mathrm{loc}}\left([0,\infty),\mathbb{R}^{n}\right)$, under the smoothness assumption $\bm{\sigma}(t,\cdot)\in C^{1}_{b}(\mathbb{R}^{n})$ for all $t\geq 0$. In fact, the results in \cite{da2004invariance,buckdahn2010another} are more general in that the diffusion does not need to be affine in $\bm{u}$. However, designing a direct computational test based on these results remain challenging.

Motivated from a computational perspective, \emph{sufficient} conditions are available in terms of the {stochastic control barrier functions} (CBFs) for a.s. set invariance \cite{clark2021control,tamba2021notion,nishimura2024control}, and for quantifying the probability that a system exits a given safe set \cite{santoyo2021barrier,wang2021safety,xue2023reach,prajna2007framework}. In general, stochastic CBFs can be computationally intensive and may suffer from additional conservatism depending on their parameterization (e.g., sum-of-squares polynomial \cite{prajna2007framework,steinhardt2012finite}, neural network \cite{mathiesen2022safety}). Moreover, since the formulation provides only sufficient conditions, an infeasible search within a parameterization class yields an inconclusive result. 

Compared to the existing literature, our proposed computational framework comes from a different perspective. It builds upon Doob's $h$-transform \cite{doob1957conditional,doob1984classical}, and involves a relatively simple computation of a score-vector field (\textbf{step 1}) followed by solving a static linear system (\textbf{step 2}). The proposed computation does not involve derivatives of the diffusion coefficient $\bm{\sigma}$ or additional control as in \eqref{equivdetcontrolsystem}. Unlike the stochastic CBF-based methods, the conditions derived here are necessary and sufficient. Even if the a.s. invariance does not hold for a specific application, falsifying so using a method such as ours, can rigorously justify weaker or more conservative problem formulations.


\subsection{Contributions} 
Our specific contributions are the following.

\begin{itemize}

\item We derive constructive necessary and sufficient conditions to certify or falsify a.s. set invariance for controlled diffusions. This is done in Sec. \ref{sec:giventimesetinvariance} for $\mathcal{I}= [0,T]$, and in Sec. \ref{sec:infinitetimesetinvariance} for $\mathcal{I}= [0,\infty)$. The conditions involve the computation of score functions \cite{song2019generative,song2020score,vahdat2021score,jo2022score,lee2023score,lee2023convergence} well-known in diffusion models for generative AI. However, the use of score functions in controlled set invariance is new.

\item For $T<\infty$, we prove inverse optimality guarantee in Sec. \ref{sec:inverseoptimality} for the provably certified controller.

\item We illustrate the proposed computational test using several analytical and numerical examples in Sec. \ref{sec:numerics}. In particular, we point out how numerical computation in general can be performed for $\mathcal{I}= [0,T]$ using the Feynman-Kac path integrals, and for $\mathcal{I}= [0,\infty)$ using the inverse power iterations.

\end{itemize}

  

\section{Background}\label{sec:Background}
In this section, we collect some notations and results which find use in our derivations and proofs.

We use $\circ$ for composition of operators. We denote the Euclidean gradient and Hessian w.r.t. $\bm{x}\in\mathbb{R}^{n}$ as $\nabla_{\bm{x}}$ and $\hess_{\bm{x}}$, respectively. For $\phi\in\mathcal{C}^{1,2}\left([0,\infty);\mathbb{R}^n\right)$, the process $\bm{x}_t$ in \eqref{UnconditionedSDE} has infinitesimal generator $\gen$ given by
\begin{align}
\gen \left[\phi\right](t,\bm{x}) 
= \langle\bm{f},\nabla_{\bm{x}}\phi\rangle + \dfrac{1}{2}\langle\bm{\Sigma},\hess_{\bm{x}}\:\phi\rangle
\label{defGenerator}
\end{align}
which follows from It\^{o}'s lemma \cite[Lemma 7.3.2, Thm. 7.3.3]{oksendal2013stochastic}. Assumptions \textbf{A1} and \textbf{A2} guarantee that $\gen$ uniquely determines\footnote{i.e., assumptions \textbf{A1-A2} guarantee the solution of the \emph{martingale problem} \cite[Sec. 1.4.3]{bakry2014analysis} associated with $\gen$, which is to guarantee that for any $\phi\in\dom(\gen)$, the map $t\mapsto\phi(\bm{x}_t)$ is right-continuous and the process $\phi(\bm{x}_t) - \int_{0}^{t}\gen\left[\phi(\bm{x}_{\tau})\right]\differential\tau, t\geq 0$, is a martingale.} the diffusion process $\bm{x}_t$ with transition density $p(s,\bm{x},t,\bm{y})$ for all $0\leq s<t < \infty$ w.r.t. the Lebesgue measure on $\mathbb{R}^{n}$, and that $\gen$ is strictly elliptic \cite[Ch. 3]{gilbarg1977elliptic}.

For all $\bm{x}\in\mathbb{R}^{n}$, define the expectation operator
\begin{align}
\mathbb{E}_{\bm{x}}\left[\phi(\bm{x}_t)\right] := \displaystyle\int_{\mathbb{R}^{n}} \phi(\bm{y}) \,  p(0,\bm{x},t,\bm{y}) \, \differential\bm{y} \mid \bm{x}_0 = \bm{x}. 
\label{defExpectation}
\end{align}
The action of the transition density $p$ is then related to the generator $\gen$ as
\begin{align}
\dfrac{\partial}{\partial t}\mathbb{E}_{\bm{x}}\left[\phi(\bm{x}_t)\right] = 
\mathbb{E}_{\bm{x}}\left[\gen\left[\phi\right](t, \bm{x}_t)\right],
\label{ActionOfGeneratorL}
\end{align}
i.e., $\exp\left(t\gen\right)$ is the Markovian semigroup generated by $\gen$ for all $t\geq 0$. 

The generator $\gen$ is associated with the \emph{backward Kolmogorov PDE}
\begin{align}
\dfrac{\partial\phi}{\partial t} + \gen\left[\phi\right]=0.
\label{BackwardPDE}    
\end{align}
The adjoint of $\gen$, denoted by  $\gen^{\dagger}$, is
\begin{align}
\gen^{\dagger}[\mu](t,\bm{x}) := \nabla_{\bm{x}}\cdot\left(\mu\bm{f}\right) - \dfrac{1}{2}\Delta_{\bm{\Sigma}}\:\mu,
\label{defAdjoint}    
\end{align}
where the weighted Laplacian $\Delta_{\bm{\Sigma}} := \displaystyle\sum_{i,j=1}^{n}\dfrac{\partial^2\left(\bm{\Sigma}_{ij}(t,\bm{x})\mu\right)}{\partial x_i\partial x_j}$, and $\gen^{\dagger}$ is associated with the \emph{forward Kolmogorov a.k.a. Fokker-Planck PDE}
\begin{align}
\dfrac{\partial \mu}{\partial t} + \gen^{\dagger}\left[\mu\right] = 0,
\label{FordwardPDE}    
\end{align}
that propagates the law of the initial state $\bm{x}_0 \sim \mu_0$ forward in time, i.e., $\bm{x}_t \sim \mu(t,\bm{x})$ obtained as solution of the PDE \eqref{FordwardPDE} subject to initial condition $\mu(t=0,\bm{x})=\mu_0$. In particular, $\differential\mu(t,\bm{x})=\int p(t_0,\bm{x}_{0},t,\bm{x})\differential\mu_0(\bm{x}_0)$, and the transition probability density $p$ solves the PDE initial value problem
\begin{align}
\dfrac{\partial p}{\partial t} + \gen^{\dagger}\left[p\right] = 0, \quad p(t_0,\bm{x}_0,t_{0},\bm{x})=\delta(\bm{x}-\bm{x}_0),
\label{TransitionPDFIVP}    
\end{align}
where $\delta(\cdot)$ denotes the Dirac delta. In the special case $\bm{f}\equiv \bm{0}$, $\bm{\sigma}=\bm{I}_{n}$, the PDE \eqref{BackwardPDE} (resp. \eqref{FordwardPDE}) reduces to the backward (resp. forward) heat equation.

\begin{definition}\label{def:carreduchamp}
The \emph{carr\'{e} du champ operator} $\Gamma$ \cite{kunita1969absolute,ledoux2000geometry} associated with the infinitesimal generator $\gen$, is
\begin{align}
\Gamma\left(\phi,\psi\right) := \dfrac{1}{2}\left(\gen\left[\phi\psi\right] -\phi \gen\left[\psi\right] - \psi \gen\left[\phi\right]\right)
\label{defCarreDuChamp}
\end{align}
defined for all $(\phi,\psi)\in\mathcal{A}\subseteq\dom\left(\gen\right)$, where $\mathcal{A}$ is a vector subspace in the domain of $\gen$ such that the product $\phi\psi\in\dom\left(\gen\right), \:\forall(\phi,\psi)\in\mathcal{A}$. In other words, $\mathcal{A}$ is an algebra.
\end{definition}
By definition, $\Gamma$ is symmetric, bilinear and nonnegative (since $\Gamma(\phi,\phi)\geq 0$ for all $\phi\in\mathcal{A}$). Definition \ref{def:carreduchamp} and Lemma \ref{LemmaGammaForOurGenerator} stated next will be useful in Theorem \ref{Thm:GivenFiniteTimeInvariance} statement (Sec. \ref{subsec:conditioneddiffusion}) and in its proof (Appendix \ref{App:ProofThm:GivenFiniteTimeInvariance}).
\begin{lemma}[\textbf{Carr\'{e} du champ for It\^{o} diffusion}]\label{LemmaGammaForOurGenerator}
For the generator $\gen$ in \eqref{defGenerator} and $(\phi,\psi)$ as in Definition \ref{def:carreduchamp}, the carr\'{e} du champ operator \eqref{defCarreDuChamp} is given by
\begin{align}
\Gamma\left(\phi,\psi\right) = \dfrac{1}{2}\langle\bm{\Sigma}\nabla\phi,\nabla\psi\rangle.
\label{GammaInTermsOfSigma}
\end{align}
Notably, $\Gamma$ is independent of the drift coefficient $\bm{f}$.
\end{lemma}
\begin{proof}
Using \eqref{defGenerator} in \eqref{defCarreDuChamp}, we get
\begin{align}
&\Gamma\left(\phi,\psi\right)= \dfrac{1}{4}\langle\bm{\Sigma},\hess_{\bm{x}}\left(\phi\psi\right) - \phi\hess_{\bm{x}}\psi - \psi\hess_{\bm{x}}\phi\rangle\nonumber\\
&=\dfrac{1}{4}\bigg\langle\!\bm{\Sigma}, \left(\nabla_{\bm{x}}\phi\right)\left(\nabla_{\bm{x}}\psi\right)^{\top}\!\bigg\rangle + \dfrac{1}{4}\bigg\langle\!\bm{\Sigma}, \left(\nabla_{\bm{x}}\psi\right)\left(\nabla_{\bm{x}}\phi\right)^{\top}
\!\bigg\rangle
\label{SimplifyGamma}
\end{align}
where the last equality uses the identity: $\hess_{\bm{x}}\left(\phi\psi\right) = \phi\hess_{\bm{x}}\psi + \left(\nabla_{\bm{x}}\phi\right)\left(\nabla_{\bm{x}}\psi\right)^{\top} + \psi\hess_{\bm{x}}\phi + \left(\nabla_{\bm{x}}\psi\right)\left(\nabla_{\bm{x}}\phi\right)^{\top}$.

Recalling that $\bm{\Sigma}$ is symmetric and using the definition of the Hilbert-Schmidt inner product $\langle\cdot,\cdot\rangle$, we have
\begin{subequations}
\label{FirstInnerProduct}
\begin{align}
\bigg\langle\bm{\Sigma}, \left(\nabla_{\bm{x}}\phi\right)\left(\nabla_{\bm{x}}\psi\right)^{\top}\bigg\rangle &= \tr\left(\bm{\Sigma}\left(\nabla_{\bm{x}}\phi\right)\left(\nabla_{\bm{x}}\psi\right)^{\top}\right) \nonumber\\
&= \tr\left(\left(\nabla_{\bm{x}}\psi\right)^{\top}\bm{\Sigma}\left(\nabla_{\bm{x}}\phi\right)\right)\label{CycPerm1}\\
&=\left(\nabla_{\bm{x}}\phi\right)^{\top}\bm{\Sigma}\left(\nabla_{\bm{x}}\psi\right)\label{Transpose1},
\end{align}
\end{subequations}
where we used the invariance of trace under cyclic permutation (in \eqref{CycPerm1}) and under transposition (in \eqref{Transpose1}). Likewise, 
\begin{align}
\bigg\langle\bm{\Sigma}, \left(\nabla_{\bm{x}}\psi\right)\left(\nabla_{\bm{x}}\phi\right)^{\top}\bigg\rangle &= \tr\left(\bm{\Sigma}\left(\nabla_{\bm{x}}\psi\right)\left(\nabla_{\bm{x}}\phi\right)^{\top}\right)\nonumber\\
&= \left(\nabla_{\bm{x}}\phi\right)^{\top}\bm{\Sigma}\left(\nabla_{\bm{x}}\psi\right),
\label{SecondInnerProduct}
\end{align}
where \eqref{SecondInnerProduct} again used the invariance of trace under cyclic permutation. Substituting \eqref{FirstInnerProduct}-\eqref{SecondInnerProduct} in \eqref{SimplifyGamma}, we arrive at \eqref{GammaInTermsOfSigma}.
\end{proof}


\section{Controlled Set Invariance for $\mathcal{I}=[0,T]$}\label{sec:giventimesetinvariance}
In Sec. \ref{subsec:conditioneddiffusion}, we begin by determining what the conditioned diffusion process for the prior dynamics \eqref{UnconditionedSDE} should be, so as to satisfy the desired set invariance. In Sec. \ref{subsec:finitetimeScoreVectorFieldAndController}, using this conditioned diffusion, we ``back out" to characterize the controllers or the lack thereof for the controlled process \eqref{ControlledSDE}.

\subsection{Conditioned diffusion}\label{subsec:conditioneddiffusion}
Our main idea behind prescribed finite time set invariance is to condition the unforced process \eqref{UnconditionedSDE} to not exit the given $\mathcal{X}$ for all $t\in\mathcal{I}$. This conditioning is done via a version of the \emph{Doob's $h$ transform} \cite[Ch. IV.39]{rogers2000diffusions}, \cite{pinsky1985convergence, deblassie1988doob,chetrite2015nonequilibrium} that originated in the works of Doob \cite{doob1957conditional,doob1984classical} for conditioning Wiener processes. 

The $h$ in Doob's $h$ transform stands for ``harmonic" as the conditioning involves constructing suitable positive harmonic function w.r.t. the generator of the unforced process. The sample path dynamics and the generator of the space-time conditioned process needed in our setting, are summarized in Theorem \ref{Thm:GivenFiniteTimeInvariance}. Before presenting this result, we formalize the necessary ideas.

Consider the filtered probability space $\left(\Omega,\mathcal{F},\left(\mathcal{F}_{t}\right)_{t\geq 0},\mathbb{P}\right)$ associated with the unforced Markov diffusion process $\bm{x}_{t}$ in \eqref{UnconditionedSDE} with initial condition $\bm{x}_{0}$, where $\bm{f},\bm{\sigma}$ satisfy \textbf{A1}, \textbf{A2}. Given $\mathcal{X}\subset\mathbb{R}^{n}$ satisfying \textbf{A3}, let the \emph{first exit time}
\begin{align}
\tau_{\mathcal{X}} := \inf\{ t\geq 0 \mid \bm{x}_t\notin\mathcal{X}\}.
\label{deffFirstExitTime}
\end{align}
Given $T<\infty$ and an event $A\in\mathcal{F}_{t}$, we define the conditional probability 
\begin{align}
\widetilde{\mathbb{P}}\left(A\right):=\mathbb{E}\left[\mathbf{1}_{A}\mid \tau_{\mathcal{X}}>T\right],
\label{defCondProbOfEvent}
\end{align}
where $\mathbf{1}$ denotes the indicator function, and the expectation $\mathbb{E}$ is w.r.t. the measure $\mathbb{P}$. For any $0\leq t < T$, letting
\begin{align}
&\mathbb{P}_{\bm{x}}(\cdot):= \mathbb{P}\left(\cdot\mid\bm{x}_{0}=\bm{x}\right), \quad \widetilde{\mathbb{P}}_{\bm{x}}(\cdot):= \widetilde{\mathbb{P}}\left(\cdot\mid\bm{x}_{0}=\bm{x}\right),\label{defPtildeP}\\
&h_{T}(t,\bm{x}) := \mathbb{P}_{\bm{x}}\left(\tau_{\mathcal{X}}>T\right),
\label{defhT}
\end{align}
we have $\widetilde{\mathbb{P}}_{\bm{x}}\ll\mathbb{P}_{\bm{x}}$, and
\begin{align}
\widetilde{\mathbb{P}}_{\bm{x}}\left(A\right) = \dfrac{\mathbb{P}_{\bm{x}}\left(A,\tau_{\mathcal{X}}>T\right)}{\mathbb{P}_{\bm{x}}\left(\tau_{\mathcal{X}}>T\right)}=\dfrac{\mathbb{E}_{\bm{x}}\left[\mathbf{1}_{A}h_{T}\left(\bm{x}_{t\wedge\tau_{\mathcal{X}}}\right)\right]}{h_{T}(t,\bm{x})},
\label{ConditionalProb}
\end{align}
wherein we used the Markov property and the notation $a\wedge b := \min\{a,b\}\:\forall a,b\in\mathbb{R}$. A consequence of definition \eqref{defhT} is that the function $h_{T}(\cdot,\cdot)$ takes values in $[0,1]$. The subscript $T$ in notation $h_{T}(\cdot,\cdot)$ emphasizes the function's parametric dependence on the given horizon length $T$.  

Restricted to $\mathcal{F}_{t}$, the law of the unconditioned process $\bm{x}_{t}$ in \eqref{UnconditionedSDE} with initial condition $\bm{x}_{0}=\bm{x}$ is $\mathbb{P}_{\bm{x}}\vert_{\mathcal{F}_{t}}$. With the same initial condition, let $\widetilde{\bm{x}}_{t}$ be the conditioned diffusion process associated with the law $\widetilde{\mathbb{P}}_{\bm{x}}\vert_{\mathcal{F}_{t}}$. From \eqref{ConditionalProb}, the Radon-Nikodym derivative
\begin{align}
\dfrac{\differential\widetilde{\mathbb{P}}_{\bm{x}}}{\differential\mathbb{P}_{\bm{x}}}\bigg\vert_{\mathcal{F}_{t}} = \dfrac{h_{T}\left(\bm{x}_{t\wedge\tau_{\mathcal{X}}}\right)}{h_{T}(t,\bm{x})},
\label{RadonNikodym}    
\end{align}
which will be used in the proof of Theorem \ref{Thm:GivenFiniteTimeInvariance} to derive the sample path dynamics for $\widetilde{\bm{x}}_{t}$.

In proving Theorem \ref{Thm:GivenFiniteTimeInvariance}, we will also make use of the fact that the LHS of the backward Kolmogorov PDE \eqref{BackwardPDE} can be seen as the generator for the space-time lifting of \eqref{UnconditionedSDE} to subsets of $\mathbb{R}^{n+1}$, i.e., the generator for the lifted SDE
\begin{align}
\begin{pmatrix}
\differential\bm{x}_t\\
\differential t
\end{pmatrix} = \begin{pmatrix}
\bm{f}(t,\bm{x}_t)\\
1
\end{pmatrix}\differential t + \begin{pmatrix}
\bm{\sigma}(t,\bm{x}_t)\\
\bm{0}_{1\times p}
\end{pmatrix}\differential\bm{w}_t.
\label{SpaceTimeUnconditionedSDE}
\end{align}
Indeed, for $\phi\in\mathcal{C}^2\left(\mathbb{R}^{n+1}\right)$, the generator $\gen_0$ of \eqref{SpaceTimeUnconditionedSDE} is 
\begin{align}
&\gen_{0}\left[\phi\right]\begin{pmatrix}
\bm{x}\\
t
\end{pmatrix} 
= 
\Bigg\langle 
\begin{pmatrix}
\bm{f}\\
1
\end{pmatrix},\begin{pmatrix}
\nabla_{\bm{x}}\phi\\
\frac{\partial\phi}{\partial t}
\end{pmatrix}
\Bigg\rangle 
\nonumber\\
&+ 
\dfrac{1}{2}
\Bigg\langle\!\!\!\begin{pmatrix}
\!\bm{\Sigma} & \bm{0}_{n\times 1}\\
\bm{0}_{1\times n} & 0
\end{pmatrix},{\small{\begin{psmallmatrix}
\hess_{\bm{x}}\phi & \begin{bmatrix}
\frac{\partial^2\phi}{\partial x_1\partial t}\\
\vdots\\
\frac{\partial^2\phi}{\partial x_n\partial t}
\end{bmatrix}\\
\begin{bsmallmatrix}
\frac{\partial^2\phi}{\partial t\partial x_1}
\hdots \frac{\partial^2\phi}{\partial t\partial x_n}
\end{bsmallmatrix} & \frac{\partial^2\phi}{\partial t^2}
\end{psmallmatrix}}}
\Bigg\rangle\nonumber\\
&= \langle\bm{f},\nabla_{\bm{x}}\phi\rangle + \dfrac{\partial\phi}{\partial t} + \dfrac{1}{2}\langle\bm{\Sigma},\hess_{\bm{x}}\phi\rangle = \gen\left[\phi\right] + \dfrac{\partial\phi}{\partial t}.
\label{GeneratorOfLiftedSDE}
\end{align}
Thus, the lifted generator $\gen_0$ and the base generator $\gen$ are related as $\gen_0 = \frac{\partial}{\partial t} + \gen$. These ideas lead to the following result (proof in Appendix \ref{App:ProofThm:GivenFiniteTimeInvariance}).

\begin{theorem}\label{Thm:GivenFiniteTimeInvariance}(\textbf{Conditioned process for prescribed finite time set invariance})
Consider an It\^{o} diffusion $\bm{x}_t$ satisfying \eqref{UnconditionedSDE} with associated generator $\gen$ given by \eqref{defGenerator}, where $\bm{f},\bm{\sigma}$ satisfy assumptions \textbf{A1}, \textbf{A2}. Given $T<\infty$ and the open set $\mathcal{X}\subset\mathbb{R}^{n}$ satisfying \textbf{A3}, the process $\bm{x}_t$ starting from an initial condition $\bm{x}_{0}\in\mathcal{X}$ and conditioned to stay in $\mathcal{X}$ a.s. $\forall t\in[0,T]$, satisfies It\^{o} diffusion
\begin{align}
\differential\widetilde{\bm{x}}_t &= \left(\bm{f}\left(t,\widetilde{\bm{x}}_t\right) + \bm{\Sigma}\left(t,\widetilde{\bm{x}}_t\right)\nabla_{\widetilde{\bm{x}}_t}\log h_{T}(t,\widetilde{\bm{x}}_t)\right)\differential t \nonumber\\
&\qquad\qquad\qquad+ \bm{\sigma}\left(t,\widetilde{\bm{x}}_t\right)\differential \bm{w}_t, \quad \widetilde{\bm{x}}_0 := \bm{x}_0,
\label{FiniteTimeSDE}    
\end{align}
where $h_T\in\mathcal{C}^{1,2}\left([0,T];\mathcal{X}\right)$ is the unique space-time harmonic function solving the Dirichlet BVP
\begin{subequations}
\begin{align}
&\gen_{0} [h_T] = \dfrac{\partial h_T}{\partial t} + \gen\left[h_T\right]=0 \quad\forall\:\left(t,\bm{x}\right)\in[0,T]\times\mathcal{X},\label{finiteTimePDE}\\
& h_T(t=T,\bm{x}) = 
1 \quad\forall \bm{x}\in\mathcal{X},
\label{InitialCondition}\\
&h_T(t,\bm{x}) = 0 \quad\forall\:(t,\bm{x})\in[0,T)\times\partial\mathcal{X}.\label{finiteTimeBC}
\end{align}
\label{finiteTimeDirichletPDEBVP}
\end{subequations}
Furthermore, $\widetilde{\bm{x}}_t$ has generator $\widetilde{\gen}$ that satisfies the backward Kolmogorov PDE: 
\begin{subequations}
\begin{align}
&\left(\dfrac{\partial}{\partial t}+\widetilde{\gen}\right)[\phi] = \left(h_T^{-1}\circ \left(\dfrac{\partial}{\partial t}+\gen\right) \circ h_T\right) \left[\phi\right] = 0,\label{NewGenAsComposition}\\ 
&\Leftrightarrow\dfrac{\partial\phi}{\partial t} + \underbrace{\gen\left[\phi\right] + 2\Gamma\left(\log h_T, \phi\right)}_{:= \widetilde{\gen}[\phi]} = 0,
\label{NewGenAsGamma}
\end{align}
\label{NewGenFiniteInTermsOfOldGen}   \end{subequations}
where the carr\'{e} du champ operator $\Gamma$ is defined in \eqref{defCarreDuChamp}.
\end{theorem}

\begin{remark}[\textbf{Zero exit probability}]\label{remark:setinvariancefinitetime}
It follows from Theorem \ref{Thm:GivenFiniteTimeInvariance} that the It\^{o} diffusion process $\widetilde{\bm{x}}_t$ in \eqref{FiniteTimeSDE} satisfies
$\widetilde{\mathbb{P}}_{\bm{x}} \left( \tau_{\mathcal{X}} \leq T \right) = 0$, i.e., has exit probability w.r.t. set $\mathcal{X}$ equal to zero up until the given $T<\infty$.
\end{remark}

\begin{remark}[\textbf{Terminal hitting of a subset of $\mathcal{X}$}]\label{remark:hittingtargetset}
Theorem \ref{Thm:GivenFiniteTimeInvariance} statement and its proof generalize mutatis mutandis for the case when the RHS of the boundary condition \eqref{InitialCondition} is replaced by an indicator function over a terminal or goal set $\mathcal{X}_{T}\subseteq\mathcal{X}$ such that $\mathcal{X}_{T}$ satisfies \textbf{A3}. In that case, $\widetilde{\bm{x}}_{t}$ in \eqref{FiniteTimeSDE} is guaranteed to stay in $\mathcal{X}\:\forall t\in[0,T]$ and hit $\mathcal{X}_{T}$ at $t=T$ a.s. Then \eqref{finiteTimePDE}, \eqref{finiteTimeBC} and \eqref{NewGenFiniteInTermsOfOldGen} remain unchanged while \eqref{InitialCondition} is replaced by
\begin{align}
h_T(t=T,\bm{x}) = 
1 \quad\forall \bm{x}\in\mathcal{X}_{T}.
\label{InitialConditionGoalSet}    \end{align}
In the special case $\mathcal{X}_{T}=\mathcal{X}$, condition \eqref{InitialConditionGoalSet} reduces to \eqref{InitialCondition}.
\end{remark}

\begin{remark}[\textbf{Transition probability density for the conditioned process}]
The result \eqref{NewGenFiniteInTermsOfOldGen} together with \eqref{defAdjoint} and \eqref{TransitionPDFIVP}, imply that the transition density for the conditioned process $\widetilde{p}_{T}$ is related to that of the unconditioned process, denoted as $p$, as follows:
\begin{align}
\widetilde{p}_{T}\left(t_{0},\bm{x}_{0},t,\bm{x}\right) = \dfrac{h_{T}(t,\bm{x})}{h_{T}(t_0,\bm{x}_0)}p\left(t_{0},\bm{x}_{0},t,\bm{x}\right),
\label{ConditionedTransitionDensity}    
\end{align}
where $0\leq t_{0} \leq t \leq T$. We will use \eqref{ConditionedTransitionDensity} in \textbf{Example 8}.
\end{remark}

In Sec. \ref{subsec:numerical_FiniteHorizon}, we will exemplify the solution of the Dirichlet BVP \eqref{finiteTimeDirichletPDEBVP} using both semi-analytical and numerical methods. We next describe how this solution plays a central role in certifying or falsifying controlled set invariance.


\subsection{From conditioned diffusion to controller}\label{subsec:finitetimeScoreVectorFieldAndController}
The conditioned diffusion \eqref{FiniteTimeSDE} differs from the prior dynamics \eqref{UnconditionedSDE} only by an additional drift term
\begin{align}
\bm{s}_{T}\left(t,\bm{x}\right) := \bm{\Sigma}\left(t,\bm{x}\right)\nabla_{\bm{x}}\log h_{T}(t,\bm{x}) \quad\forall (t,\bm{x})\in [0,T]\times \mathcal{X}.
\label{defScoreFiniteTime}    
\end{align}
We refer to $\bm{s}_{T}$ as the \emph{score vector field}. This terminology is motivated by that the RHS of \eqref{defScoreFiniteTime} is a scaled gradient of the log-likelihood since $h_{T}$, by definition \eqref{defhT}, quantifies a probability. We note that score-based computation is now widely used in generative AI models \cite{song2019generative,song2020score,vahdat2021score,jo2022score,lee2023score,lee2023convergence}. In our context, solving for $h_{T}$ from the Dirichlet BVP \eqref{finiteTimeDirichletPDEBVP} followed by computing \eqref{defScoreFiniteTime}, completes the \textbf{step 1} mentioned in Sec. \ref{sec:introduction}. 

We notice that the conditioned diffusion \eqref{FiniteTimeSDE} can be realized in the form \eqref{ControlledSDE} \emph{if and only if the score vector field $\bm{s}_{T}$ in \eqref{defScoreFiniteTime} is in the range of the input matrix $\bm{G}$}. This allows us to constructively perform \textbf{step 2} mentioned in Sec. \ref{sec:introduction}, i.e., to certify or falsify controllers for the desired set invariance, as summarized in the following.
\begin{theorem}[\textbf{Controller for prescribed finite time set invariance}]\label{Thm:ControllerFiniteTime}
Given $T<\infty$, the drift and diffusion coefficients coefficients $\bm{f},\bm{\sigma}$ satisfying \textbf{A1}, \textbf{A2}, the sets $\mathcal{X},\mathcal{X}_{T}\subseteq \mathcal{X}\in\mathbb{R}^{n}$ satisfying \textbf{A3}, let $h_{T}$ be the unique solution of the Dirichlet BVP \eqref{finiteTimePDE}, \eqref{finiteTimeBC}, \eqref{InitialConditionGoalSet}. Let $\bm{s}_{T}$ be given by \eqref{defScoreFiniteTime}, and let $\mathcal{R}(\cdot)$ denote ``the range of". The necessary and sufficient condition for the controlled diffusion \eqref{ControlledSDE} to meet controlled set invariance with problem data $\left(\bm{f},\bm{\sigma},\mathcal{X},\mathcal{X}_{T},\mathcal{I}=[0,T]\right)$ is 
\begin{align}
\bm{s}_{T}\in\mathcal{R}\left(\bm{G}\right)\quad\forall(t,\bm{x})\in\mathcal{I}\times\mathcal{X}.
\label{InTheRangeFiniteTime}    
\end{align}
\end{theorem}
Violation of \eqref{InTheRangeFiniteTime}, i.e., the statement $$\exists(t,\bm{x})\in\mathcal{I}\times\mathcal{X}\quad\text{such that}\quad\bm{s}_{T}\notin\mathcal{R}\left(\bm{G}\right),$$ 
is equivalent to non-existence of Markovian control \eqref{MarkovianControl} to achieve the a.s. set invariance with problem data $\left(\bm{f},\bm{\sigma},\mathcal{X},\mathcal{X}_{T},\mathcal{I}=[0,T]\right)$. Thus, Theorem \ref{Thm:ControllerFiniteTime} can be used for falsifying controlled set invariance for given problem data. 

On the other hand, when \eqref{InTheRangeFiniteTime} holds, Theorem \ref{Thm:ControllerFiniteTime}  not only certifies the existence of a Markovian controller \eqref{MarkovianControl} achieving the desired set invariance, but it also characterizes all such controllers $\bm{u}$ as the solution set of the static linear system 
\begin{align}
\bm{G}(t,\bm{x})\bm{u}(t,\bm{x})=\bm{s}_{T}(t,\bm{x}) \stackrel{\eqref{defScoreFiniteTime}}{=} \bm{\Sigma}\left(t,\bm{x}\right)\nabla_{\bm{x}}\log h_{T}(t,\bm{x}).
\label{StaticLinearSystem}
\end{align}
For instance, this implies that when $\bm{G}$ is wide ($n<m$) and has full row rank $\forall(t,\bm{x})\in\mathcal{I}\times\mathcal{X}$, then the set of all certified controllers are of the form 
$$\bm{u}_{\rm{part}}(t,\bm{x}) + \bm{v}(t,\bm{x}),$$ for a particular (e.g., minimum norm) solution $\bm{u}_{\rm{part}}$ satisfying $\bm{G}(t,\bm{x})\bm{u}_{\rm{part}}(t,\bm{x})=\bm{s}_{T}(t,\bm{x})$, and any $\bm{v}(t,\bm{x})\in\mathcal{N}\left(\bm{G}(t,\bm{x})\right)$, the nullspace of $\bm{G}(t,\bm{x})$.



\section{Controlled Set Invariance for $\mathcal{I}=
[0,\infty)$}\label{sec:infinitetimesetinvariance}

\subsection{Conditioned diffusion}\label{subsec:InfinitetimeConditionedDiffusion}
To generalize the prescribed finite time controlled set invariance results to $T\uparrow \infty$, the main idea is to write the $h_T$ in \eqref{finiteTimeDirichletPDEBVP} as an asymptotic expansion in that limit. Specifically,
\begin{align}
h_{\infty}(t,\bm{x}):= \displaystyle\lim_{T\uparrow\infty}h_{T}\left(t,\bm{x}\right)=e^{\lambda_0t}\psi_0(\bm{x})+ o\left(e^{\lambda_0t}\right)
\label{AsymptoticExpansion}    
\end{align}
uniformly in $\mathcal{X}$, where $\left(\lambda_0,\psi_0\right)$ is the \emph{principal Dirichlet eigenvalue-eigenfunction pair} for the operator ($-\gen$). 

In general, the \emph{Dirichlet eigen-problem} for ($-\gen$) is
\begin{subequations}
\begin{align}
&\left(-\gen\right)[\psi(\bm{x})] = \lambda \psi(\bm{x}) \quad\forall\:\bm{x}\in\mathcal{X},\label{infiniteTimePDE}\\
&\psi(\bm{x})=0, \quad\forall \bm{x}\in \partial\mathcal{X},\label{infiniteTimeBC}
\end{align}
\label{infiniteTimeDirichletEigProb}
\end{subequations}
wherein $(\lambda,\psi)$ is the corresponding eigenvalue-eigenfunction pair. Under assumptions \textbf{A1}-\textbf{A3}, the following facts about the eigen-problem \eqref{infiniteTimeDirichletEigProb} are well-known.
\begin{itemize}
\item[(i)] Problem \eqref{infiniteTimeDirichletEigProb} admits at most countable spectrum.

\item[(ii)] For \eqref{infiniteTimeDirichletEigProb}, there exists an eigenvalue $\lambda_0\in\mathbb{R}$, called \emph{principal eigenvalue}, such that for any $\lambda\in\mathbb{C}$ solving \eqref{infiniteTimeDirichletEigProb}, we have
\begin{align}
\lambda_0 \leq {\mathrm{Real}}\left(\lambda\right).
\label{defPrincipalEigValue}    
\end{align}

\item[(iii)] The principal eigenvalue $\lambda_0\in\mathbb{R}$ is simple.

\item[(iv)] The principal eigenfunction $\psi_0(\bm{x})>0$ for all $\bm{x}\in\mathcal{X}$.

\item[(v)] The function $\psi_0\in H^{1}_{0}\left(\mathcal{X}\right):= \{\psi\in L^{2}\left(\mathcal{X}\right) \mid \psi(\bm{x}) = 0\:\forall\bm{x}\in\partial\mathcal{X}, D^{1}\psi\in L^{2}\left(\mathcal{X}\right)\}$, the Sobolev space of
square integrable functions on $\mathcal{X}$ which vanish on $\partial\mathcal{X}$, and whose first order weak derivatives w.r.t. $\bm{x}\in\mathcal{X}$, denoted by $D^{1}$, are square integrable on $\mathcal{X}$.

\end{itemize}
For (i), see \cite[Ch. 6.2, Thm. 5]{evans2022partial}, \cite[Thm. 8.6]{gilbarg1977elliptic}. For (ii)-(iv), see \cite[Ch. 6.5.2, Thm. 3]{evans2022partial}. For bounded smooth $\mathcal{X}$, these properties follow from the Krein-Rutman theory \cite{krein1948linear}. They have been extended for $\mathcal{X}$ bounded non-smooth \cite{berestycki1994principal}, or unbounded \cite{berestycki2015generalizations}. For (v), see \cite{gilbarg1977elliptic}.

A consequence of \eqref{AsymptoticExpansion} is that
\begin{align}
\nabla_{\bm{x}}\log h_{\infty}(t,\bm{x})= \nabla_{\bm{x}}\log\psi_{0}\left(\bm{x}\right),
\label{LimitOfGradLog}
\end{align}
and thanks to the facts (iv)-(v), $\nabla_{\bm{x}}\log\psi_{0}$ is well-defined.
Note in particular that even though $h_{\infty}$ depends on $(t,\bm{x})$, its logarithmic gradient has no explicit dependence on $t$. The next result shows that \eqref{LimitOfGradLog} determines the conditioned diffusion process that renders $\mathcal{X}$ invariant over time interval $\mathcal{I}=[0,\infty)$.

\begin{theorem}\label{Thm:GivenInFiniteTimeInvariance}(\textbf{Conditioned process for infinite time set invariance})
Consider an It\^{o} diffusion $\bm{x}_t$ satisfying \eqref{UnconditionedSDE} with associated generator $\gen$ given by \eqref{defGenerator}, where $\bm{f},\bm{\sigma}$ satisfy assumptions \textbf{A1}, \textbf{A2}. Given $\mathcal{X}\subset\mathbb{R}^{n}$ satisfying \textbf{A3}, the process $\bm{x}_t$ starting from an initial condition $\bm{x}_{0}\in\mathcal{X}$ and conditioned to stay in $\mathcal{X}$ a.s. $\forall t\in[0,\infty)$, is the It\^{o} diffusion
\begin{align}
\differential\widetilde{\bm{x}}_t &= \left(\bm{f}\left(t,\widetilde{\bm{x}}_t\right) + \bm{\Sigma}\left(t,\widetilde{\bm{x}}_t\right)\nabla_{\widetilde{\bm{x}}_t}\log \psi_{0}(\widetilde{\bm{x}}_t)\right)\differential t \nonumber\\
&\qquad\qquad\qquad+ \bm{\sigma}\left(t,\widetilde{\bm{x}}_t\right)\differential \bm{w}_t, \quad \widetilde{\bm{x}}_0 := \bm{x}_0,
\label{InFiniteTimeSDE} \end{align}
where $\psi_0\in\mathcal{C}^{2}\left(\mathcal{X}\right)$ is the principal eigenfunction of $\left(-\gen\right)$. In other words, $\psi_0$ solves the Dirichlet eigen-problem:
\begin{subequations}
\begin{align}
&\left(-\gen\right)[\psi_0(\bm{x})] = \lambda_0 \psi_0(\bm{x}) \quad\forall\:\bm{x}\in\mathcal{X},\label{infiniteTimePDE}\\
&\psi_0(\bm{x})=0, \quad\forall \bm{x}\in \partial\mathcal{X},\label{infiniteTimeBC}
\end{align}
\label{infiniteTimeDirichletPDEBVP}
\end{subequations}
where $\lambda_0$ is the principal eigenvalue (in the sense of \eqref{defPrincipalEigValue}) for $\left(-\gen\right)$.
\end{theorem}

The proof of Theorem \ref{Thm:GivenInFiniteTimeInvariance} in Appendix \ref{App:ProofThm:InfiniteTimeInvariance} shows that the infinite horizon conditioned diffusion \eqref{InFiniteTimeSDE} arises as the
$T\uparrow\infty$ parametric limit of the fixed horizon conditioned diffusion \eqref{FiniteTimeSDE}. As a consequence, \eqref{InFiniteTimeSDE} also guarantees the a.s. invariance for the same set $\mathcal{X}$ for any fixed $T<\infty$. 

In Sec. \ref{subsec:numerical_InfiniteHorizon}, we will exemplify the solution of \eqref{infiniteTimeDirichletPDEBVP} using both analytical and numerical methods.


\subsection{From conditioned diffusion to controller}\label{subsec:infinitetimeScoreVectorFieldAndController}

Akin to the fixed $T<\infty$ case, the conditioned diffusion \eqref{InFiniteTimeSDE} differs from the prior dynamics \eqref{UnconditionedSDE} only by an additional drift \begin{align}
\bm{s}_{\infty}(t,\bm{x}) := \bm{\Sigma}\left(t,\bm{x}\right)\nabla_{\bm{x}}\log \psi_{0}(\bm{x}) \quad\forall (t,\bm{x})\in [0,\infty)\times \mathcal{X}.
\label{defScoreInfiniteTime}
\end{align}
The score vector field $\bm{s}_{\infty}$ in \eqref{defScoreInfiniteTime} is the infinite horizon analogue of $\bm{s}_{T}$ in \eqref{defScoreFiniteTime}. The infinite horizon variant of Theorem \ref{Thm:ControllerFiniteTime} follows.

\begin{theorem}[\textbf{Controller for infinite time set invariance}]\label{Thm:ControllerInfiniteTime}
Given the drift and diffusion coefficients coefficients $\bm{f},\bm{\sigma}$ satisfying \textbf{A1}, \textbf{A2}, the set $\mathcal{X}\subset\mathbb{R}^{n}$ satisfying \textbf{A3}, let $\psi_0$ be the principal Dirichlet eigenfunction solving \eqref{infiniteTimeDirichletPDEBVP}. Let $\bm{s}_{\infty}$ be given by \eqref{defScoreInfiniteTime}, and let $\mathcal{R}(\cdot)$ denote ``the range of". The necessary and sufficient condition for the controlled diffusion \eqref{ControlledSDE} to meet controlled set invariance with problem data $\left(\bm{f},\bm{\sigma},\mathcal{X},\mathcal{I}=[0,\infty)\right)$ is 
\begin{align}
\bm{s}_{\infty}\in\mathcal{R}\left(\bm{G}\right)\quad\forall(t,\bm{x})\in\mathcal{I}\times\mathcal{X}.
\label{InTheRangeInfiniteTime}    
\end{align}
\end{theorem}
Similar to the role played by Theorem \ref{Thm:ControllerFiniteTime} for $T<\infty$, Theorem \ref{Thm:ControllerInfiniteTime} can be used to provably falsify or certify a.s. infinite horizon set invariance. In the event of certification, the result is constructive in that a controller can be computed as a solution of the static linear system
\begin{align}
\bm{G}(t,\bm{x})\bm{u}(t,\bm{x})=\bm{s}_{\infty}(t,\bm{x}) \stackrel{\eqref{defScoreInfiniteTime}}{=} \bm{\Sigma}\left(t,\bm{x}\right)\nabla_{\bm{x}}\log \psi_{0}(\bm{x}).
\label{StaticLinearSystemInfinitehorizon}
\end{align}
An important difference compared to $T<\infty$ case is that the gradient in \eqref{StaticLinearSystemInfinitehorizon} depends on $\bm{x}$ but not on $t$. If $\bm{G}$ and $\bm{\sigma}$ (hence $\bm{\Sigma}$) have no explicit dependence on $t$, then \eqref{StaticLinearSystemInfinitehorizon} implies that $\bm{u}$ must be a time-invariant state feedback.

\section{Inverse Optimality}\label{sec:inverseoptimality}
So far we have shown that the proposed two step computation constructively certifies or falsifies controlled set invariance for It\^{o} diffusions. The qualifier ``constructively" means that in the certification case, a Markovian control policy is also derived. However, it is unclear \emph{a priori} if the controller thus derived, has any performance guarantee in the sense of minimizing a cost-to-go over the given time interval. 

In the following, we provide the inverse optimality guarantee that for fixed $T<\infty$, the proposed controller is in fact the minimum effort controller for a stochastic optimal control problem that we derive. The proof of Theorem \ref{Thm:InverseOptimalityFiniteHorizon} appears in Appendix \ref{App:ProofInverseOptimalityThmFiniteHorizon}.

\begin{theorem}[Inverse optimality]\label{Thm:InverseOptimalityFiniteHorizon}
Given $T<\infty$ and tuple $\left(\bm{f},\bm{G},\bm{\sigma},\mathcal{X},[0,T]\right)$ such that $\bm{f},\bm{\sigma}$ satisfy assumptions \textbf{A1}-\textbf{A2}, and $\mathcal{X}$ satisfies \textbf{A3}. Let $V(t,\bm{x})\in\mathcal{C}^{1,2}\left([0,T];\mathcal{X}\right)$ be the value function for the stochastic optimal control problem
\begin{subequations}
\begin{align}
&\underset{{\mathrm{Markovian}}\,\bm{u}:\mathcal{I}\times\mathcal{X}\mapsto\mathbb{R}^{m}}{\arg\min}\int_{0}^{T}\mathbb{E}_{\bm{x}_{t}^{\bm{u}} }\left[q\left(\bm{x}_{t}^{\bm{u}}\right)+\dfrac{1}{2}\|\bm{u}\|_2^2 \right]\:\differential t\label{StochasticOCPobjective}\\
&\qquad\,\mathrm{subject\;to}\qquad{\mathrm{controlled\; SDE}\;}\eqref{ControlledSDE},\\
&\qquad\qquad\qquad\qquad\;\;\mathrm{given}\; \bm{x}_{0}^{\bm{u}}\in\mathcal{X},\label{InverseOptimalitySetContainment}
\end{align}
\label{StochasticOCP}
\end{subequations}
where
\begin{align}
    q(\bm{x}) := \begin{cases} 0 & \textrm{ if } \bm{x} \in \mathcal{X},\\ 
    \infty & \textrm{ if } \bm{x} \in \partial\mathcal{X}.
    \end{cases}
\label{state_cost}
\end{align}
If $\bm{\Sigma}\nabla_{\bm{x}}V(t,\bm{x})$ is in the range of $\bm{G}$, and $\bm{G}\bm{G}^{\top}=\bm{\Sigma}$ for all $(t,\bm{x})\in [0,T]\times\mathcal{X}$, then the optimal control $\bm{u}^{\mathrm{opt}}=\bm{G}^{\top}\nabla_{\bm{x}}V$ guarantees a.s. controlled set invariance for \eqref{ControlledSDE} w.r.t. the set $\mathcal{X}$ over $[0,T]$.
\end{theorem}
Similar results can be established for the infinite horizon case following the developments in \cite{sheu1984stochastic}. We will not pursue them here.


\section{Numerical Examples}\label{sec:numerics}
In Sec. \ref{subsec:numerical_FiniteHorizon}, we exemplify the a.s. controlled set invariance solutions for a given $T<\infty$. In Sec. \ref{subsec:numerical_InfiniteHorizon}, we provide examples for the case $T=\infty$.

\subsection{The case $\mathcal{I}=[0,T]$}\label{subsec:numerical_FiniteHorizon}
In the following, we give four examples for given $T<\infty$. \textbf{Examples 1} and \textbf{2} are instructive as they allow semi-analytic handles for the solution of the corresponding Dirichlet BVP \eqref{finiteTimeDirichletPDEBVP}, and thus on the computation of the control. \textbf{Example 3} highlights the effect of the input matrix. \textbf{Example 4} is more general in that the corresponding BVP is solved numerically via Feynman-Kac path integral computation.

For the above examples, we illustrate the closed-loop sample paths achieving the desired controlled set invariance. For doing so, we apply Euler-Maruyama scheme to the conditioned SDE \eqref{FiniteTimeSDE} with step-size $10^{-3}$, and $h_{T}$ solving the specific instance of \eqref{finiteTimeDirichletPDEBVP}.

\noindent\textbf{Example 1 (Prescribed time hyper-rectangle invariance with hyper-rectangular target set and zero drift).}\\
\noindent Let the vectors $\bm{\ell},\bm{b}\in\mathbb{R}^{n}_{>0}$ such that $\bm{0}<\bm{b}<\bm{\ell}$ (elementwise vector inequality). We consider the case when \eqref{ControlledSDE} is a controlled Brownian motion in $\mathbb{R}^{n}$ that is conditioned to stay within the given hyper-rectangle 
$$\mathcal{X}\equiv{\rm{Rect}}\left({\bm{0},\bm{\ell}}\right):=\{\bm{x}\in\mathbb{R}^{n} \mid \bm{0}<\bm{x}<\bm{\ell}\}$$ up until a given $T<\infty$, and the target set to hit at $t=T$ is
$$\mathcal{X}_{T}\equiv{\rm{Rect}}\left({\bm{0},\bm{b}}\right):=\{\bm{x}\in\mathbb{R}^{n} \mid \bm{0}<\bm{x}<\bm{b}\}.$$
In other words, the problem data are
\begin{align}
&\left(\bm{f},\bm{G},\bm{\sigma},\mathcal{X},\mathcal{X}_{T},\mathcal{I}\right)\nonumber\\
=&\left(\bm{0},\bm{I}_{n},\bm{I}_{n},{\rm{Rect}}\left({\bm{0},\bm{\ell}}\right),{\rm{Rect}}\left({\bm{0},\bm{b}}\right),[0,T]\right).   
\label{ProblemDataExample1}
\end{align}

In this case, the Dirichlet BVP given by \eqref{finiteTimePDE}, \eqref{finiteTimeBC} and \eqref{InitialConditionGoalSet} specializes to
\begin{subequations}
\begin{align}
&\dfrac{\partial h_T}{\partial t} + \frac{1}{2}\Delta_{\bm{x}} h_T=0 \quad\forall\left(t,\bm{x}\right)\in[0,T]\times{\rm{Rect}}\left({\bm{0},\bm{\ell}}\right),\label{Example1finiteTimePDE}\\
& h_T(T,\bm{x}) = \begin{cases}1& \text{if} \quad\bm{0}<\bm{x}<\bm{b},\\
    0 & \text{otherwise}, \end{cases}
\label{Example1InitialCondition}\\
&h_T(t,\bm{x}) = 0 \quad\forall\:(t,\bm{x})\in[0,T)\times\partial{\rm{Rect}}\left({\bm{0},\bm{\ell}}\right).\label{Example1finiteTimeBC}
\end{align}
\label{Example1finiteTimeDirichletPDEBVP}
\end{subequations}
The $h_{T}$ solving \eqref{Example1finiteTimeDirichletPDEBVP} can be computed by generalizing the separation-of-variables arguments in \cite[p. 38-55]{HabermanBook}, \cite[p. 262-263]{strauss2007partial} to $n$ dimensions. Specifically, we have
\begin{subequations}
\begin{align}
&h_{T}(t,\bm{x}) =\displaystyle\sum_{m_{1}=1}^{\infty} \displaystyle\sum_{m_{2}=1}^{\infty}\hdots\displaystyle\sum_{m_{n}=1}^{\infty}C_{m_{1}m_{2}\hdots m_{n}}\prod_{i=1}^{n} d_i,\label{hTFiniteHorizonRectangle}\\
&C_{m_{1}m_{2}\hdots m_{n}} :=\exp\left(-(T-t)\right)\dfrac{2^n}{\pi^n}\scalebox{1.35}{$\displaystyle\prod_{i=1}^{n}$}\dfrac{1 - \cos\left(\dfrac{\pi m_{i} b_{i}}{l_i}\right)}{m_{i}}, \label{CFiniteHorizonRectangle}\\
&d_i :=\sin\left(\dfrac{\pi m_{i} x_{i}}{l_i}\right)\exp\left(-\dfrac{\pi^2 m_i^2}{2 l_i^2}\right). \label{dFiniteHorizonRectangle}
\end{align}
\label{FiniteHorizonRectangle}    
\end{subequations}
With the above $h_{T}$, \eqref{BackoutToControl} and \eqref{defScoreFiniteTime} yield the unique controller
\begin{align}
\bm{u}(t,\bm{x}) = \bm{s}_{T}(t,\bm{x}) = \nabla_{\bm{x}}h_{T}(t,\bm{x}).
\label{uForBrownianPrior}    
\end{align}
Equations \eqref{FiniteHorizonRectangle}-\eqref{uForBrownianPrior} thus derived, together complete the proposed computational \textbf{steps 1-2}. As explained earlier, the controller \eqref{uForBrownianPrior} is correct-by-construction in the sense it guarantees controlled set invariance for problem data \eqref{ProblemDataExample1}.

Fig. \ref{fig:1d_finite} shows the solution of the corresponding set invariance problem in $n=1$ dimension with $\ell=\pi$, $b=\pi/3$, $\mathcal{I}=[0,3]$, i.e., $T=3$. In particular, Fig. \ref{fig:1d_finite}(a) depicts the contours of $h_{T}(t,x)$ in the $(t,x)$ plane. Fig. \ref{fig:1d_finite}(b) shows 20 sample paths for the controlled SDE \eqref{FiniteTimeSDE} achieving a.s. controlled set invariance for this 1D problem data.

Fig. \ref{fig:2d_contour} and Fig. \ref{fig:2d_traj_finite} illustrate the solution of the corresponding set invariance problem in $n=2$ dimensions with $\bm{\ell} = (2,2)^{\top}$, $\bm{b} = (1,1)^{\top}$, $\mathcal{I}=[0,1]$, i.e., $T=1$. In particular, Fig. \ref{fig:2d_contour} shows the contours of $h_{T}(t,\bm{x})$. Fig. \ref{fig:2d_traj_finite} shows the sample paths for the controlled SDE \eqref{FiniteTimeSDE} achieving a.s. controlled set invariance for this 2D problem data.

\begin{figure}[t]
    \centering
    \subfigure[Contour plot of $h_{T}(t,x)$ ]{
        \includegraphics[width=0.235\textwidth]{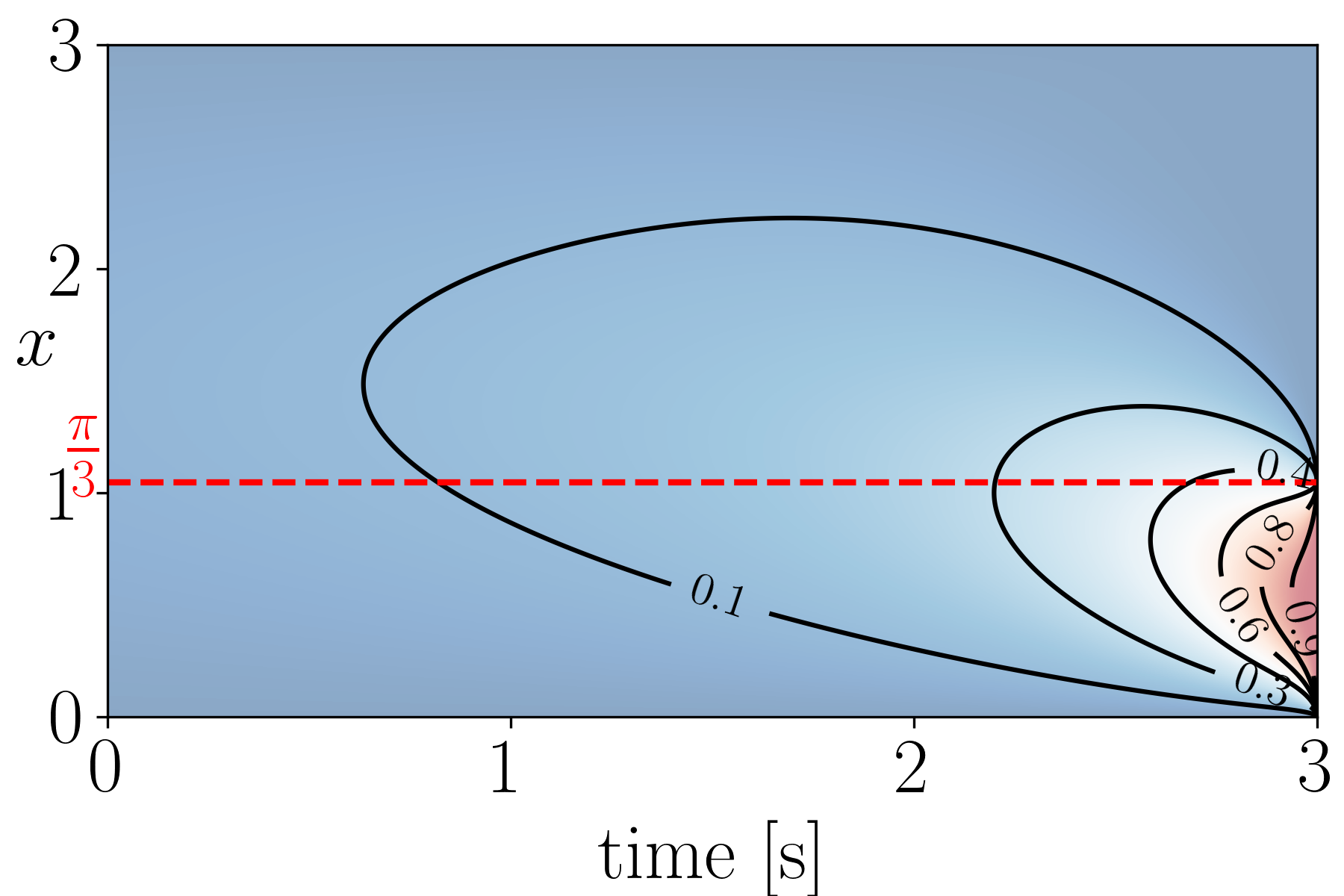}
    }
    \hspace{-11pt}
    \subfigure[20 controlled sample paths ]{
        \includegraphics[width=0.235\textwidth]{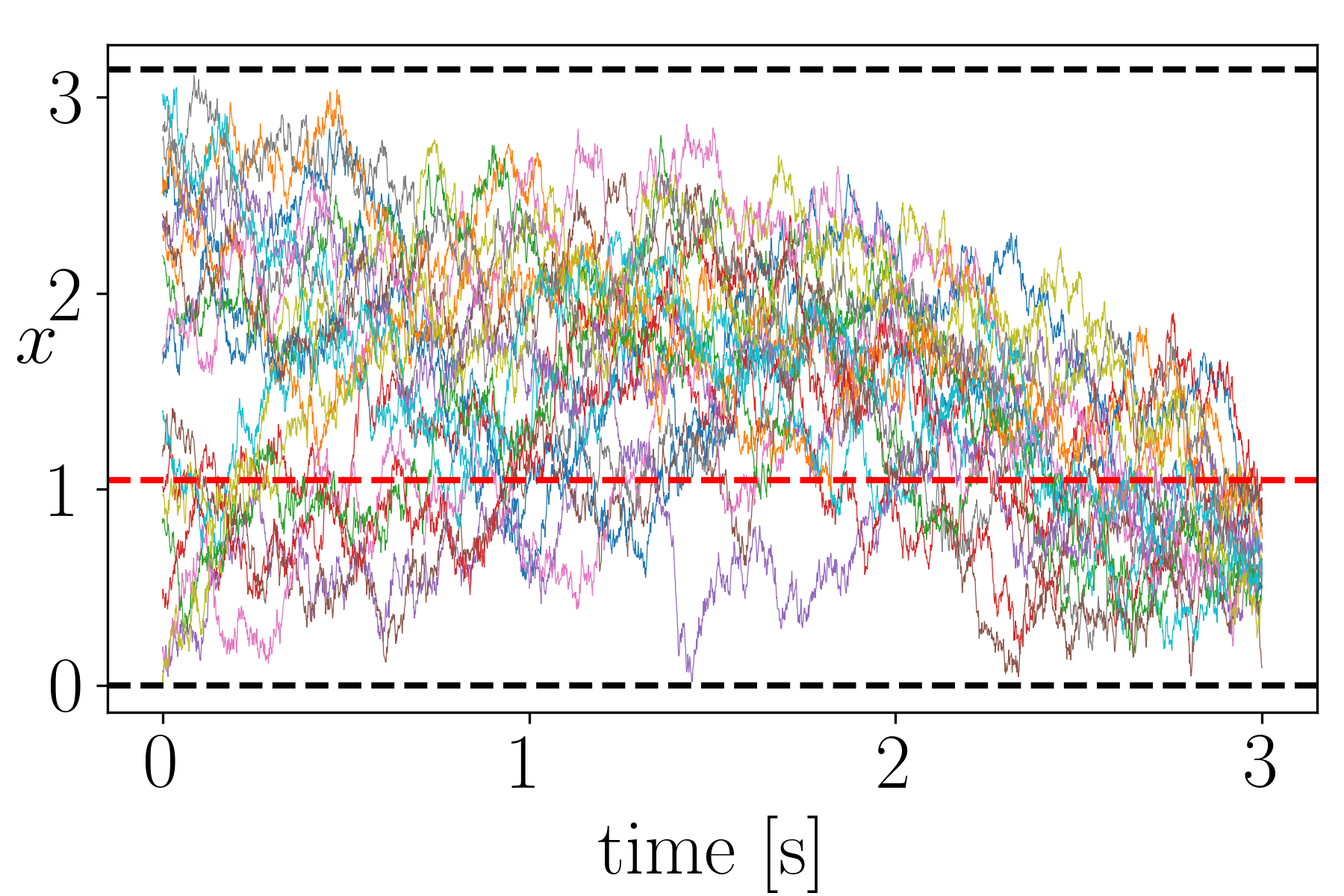}
    }
    \caption{Contour plot of $h_{T}$ and controlled sample paths for \textbf{Example 1} with $n=1$, $\ell=\pi$, $b=\pi/3$, $\mathcal{I}=[0,3]$.}
    \label{fig:1d_finite}
\end{figure}

\begin{figure*}[t]
    \centering      
    \includegraphics[width=\linewidth]{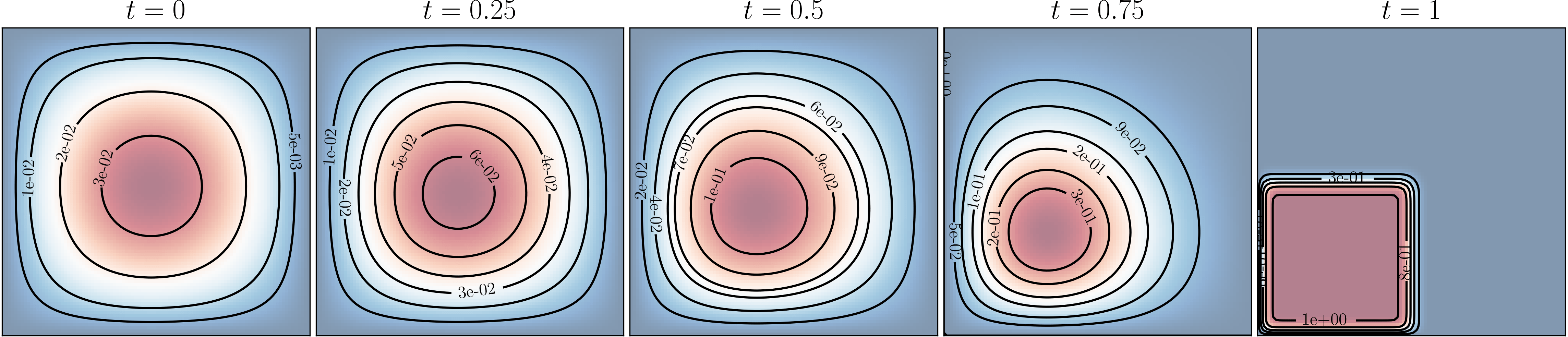}
    \caption{Contour snapshots of $h_{T}(t, \bm{x})$ for \textbf{Example 1} with $n=2$, $\bm{\ell} = (2,2)^{\top}$, $\bm{b} = (1,1)^{\top}$, $\mathcal{I}=[0,1]$. Each contour snapshot is plotted over the set $\mathcal{X}\equiv{\rm{Rect}}\left({\bm{0},\bm{\ell}}\right)$.}
    \label{fig:2d_contour}
\end{figure*}

\begin{figure}[htbp]
    \centering
    \includegraphics[width=\linewidth]{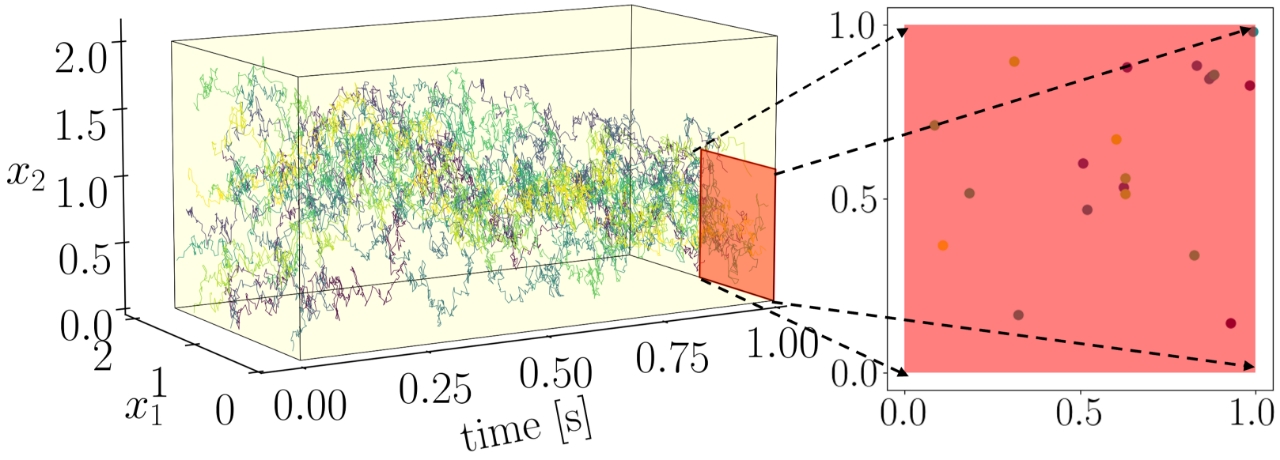}    
    \caption{20 controlled sample paths for \textbf{Example 1} with $n=2$, $\bm{\ell} = (2,2)^{\top}$, $\bm{b} = (1,1)^{\top}$, $\mathcal{I}=[0,1]$. The inset plot highlights the target set $\mathcal{X}_{T}\equiv{\rm{Rect}}\left({\bm{0},\bm{b}}\right)$ overlaid with the terminal states of the controlled sample paths.}
    \label{fig:2d_traj_finite}
\end{figure}

\noindent\textbf{Example 2 (Prescribed time disk invariance with annular target set and zero drift).}\\
\noindent For $0<r_1<r_2$, let
$$\mathcal{X}\equiv{\rm{Disk}}\left(\bm{0},r_2\right):=\{\bm{x}\in\mathbb{R}^{2} \mid \|\bm{x}\|_2 < r_2\},$$
and
$$\mathcal{X}_{T}\equiv{\rm{Annulus}}\left(r_1,r_2\right):=\{\bm{x}\in\mathbb{R}^{2} \mid r_1<\|\bm{x}\|_2 < r_2\}.$$
We consider the prior dynamics to be standard Brownian motion in $\mathbb{R}^{2}$ conditioned to render the disk $\mathcal{X}$ a.s. invariant up until given $T<\infty$, and to hit the annulus $\mathcal{X}_{T}$ at $t=T$. In other words, the problem data are
\begin{align}
&\left(\bm{f},\bm{G},\bm{\sigma},\mathcal{X},\mathcal{X}_{T},\mathcal{I}\right)\nonumber\\
=&\left(\bm{0},\bm{I}_{2},\bm{I}_{2},{\rm{Disk}}\left(\bm{0},r_2\right),{\rm{Annulus}}\left(r_1,r_2\right),[0,T]\right).
\label{ProblemDataExample2}
\end{align}

For convenience, we resort to polar coordinates, where the Dirichlet BVP given by \eqref{finiteTimePDE}, \eqref{finiteTimeBC} and \eqref{InitialConditionGoalSet} with this problem data specializes to
\begin{subequations}
\begin{align}
& \dfrac{\partial h_T}{\partial t}+\frac{1}{2}\left(\frac{\partial^2 h_T}{\partial r^2}+\frac{1}{r}\dfrac{\partial h_T}{\partial r}+\frac{1}{r^2}\frac{\partial^2 h_T}{\partial 
 \theta^2}\right)=0 \nonumber\\
 &\qquad\qquad\qquad\qquad\forall\:(t,r,\theta)\in[0,T]\times{\rm{Disk}}\left({\bm{0},r_2}\right),\label{Example2PDE}\\
 &h_T(T,r,\theta) = \begin{cases}1& \text{if} \quad(r,\theta)\in {\mathrm{Annulus}}(r_1,r_2),\\
    0 & \text{otherwise}, \end{cases}
\label{Example2InitialCondition}\\
&h_T(t,r,\theta) = 0 \quad\forall\:(t,r,\theta)\in[0,T)\times\partial{\rm{Disk}}\left({\bm{0},r_2}\right).\label{Example2finiteTimeBC}
\end{align}
\label{2DfiniteTimeDirichletPDEBVPAnnulus}
\end{subequations}
The solution for the BVP \eqref{2DfiniteTimeDirichletPDEBVPAnnulus} is (see Appendix \ref{Derivation of heat equation in annular set})
\begin{align}
    h_T(t,r,\theta)=\sum\limits_{k=1}^{\infty} c^0_{k}\phi^0_k(r)\exp\left(-\frac{1}{2}(T-t)\!\left(\dfrac{z^{0}_{k}}{r_2}\right)^{\!\!2}\right),
\label{hTExample2}
\end{align}
where $\phi^0_k(r) := J_0\left(\frac{z^{0}_{k} r}{r_2}\right)$ is the Bessel function of the first kind with order zero evaluated at $z^{0}_{k} r/r_2$, wherein \( z_k^0 \) denotes the \( k^\text{th} \) zero of \( J_0(\cdot) \); see e.g., \cite[Ch. 9.1, Ch. 9.5]{abramowitz1972handbook}. The $z_k^0$ are known to be simple and positive for all $k\in\mathbb{N}$, e.g., $z_1^{0} = 2.4048, z_2^{0} = 5.5201, z_{3}^{0} = 8.6537$, etc. The constants $c^0_{k}$ in \eqref{hTExample2} are (see Appendix \ref{Derivation of heat equation in annular set}) 
\begin{equation}\label{eq:coefficient caculation}
    c^0_k := \frac{ \int_{r_1}^{r_2} \phi^0_k(r) r\differential r }{ \int_{0}^{r_2} \left(\phi^0_k(r)\right)^2 r\differential r}.
\end{equation}
Notice that the solution \eqref{hTExample2} is radial, i.e., independent of $\theta$, as expected from the symmetry of the problem. With $h_{T}$ as in \eqref{hTExample2}, the unique correct-by-construction $\bm{u}(t,\bm{x})$ is again given by \eqref{uForBrownianPrior}. 

Fig. \ref{fig:2d_traj_annulus} compares 7 uncontrolled and controlled sample paths for problem data \eqref{ProblemDataExample2} with $r_1=1$, $r_2=2$, $T=1$.

\begin{figure}[t]
    \centering
\includegraphics[width=0.7\linewidth]{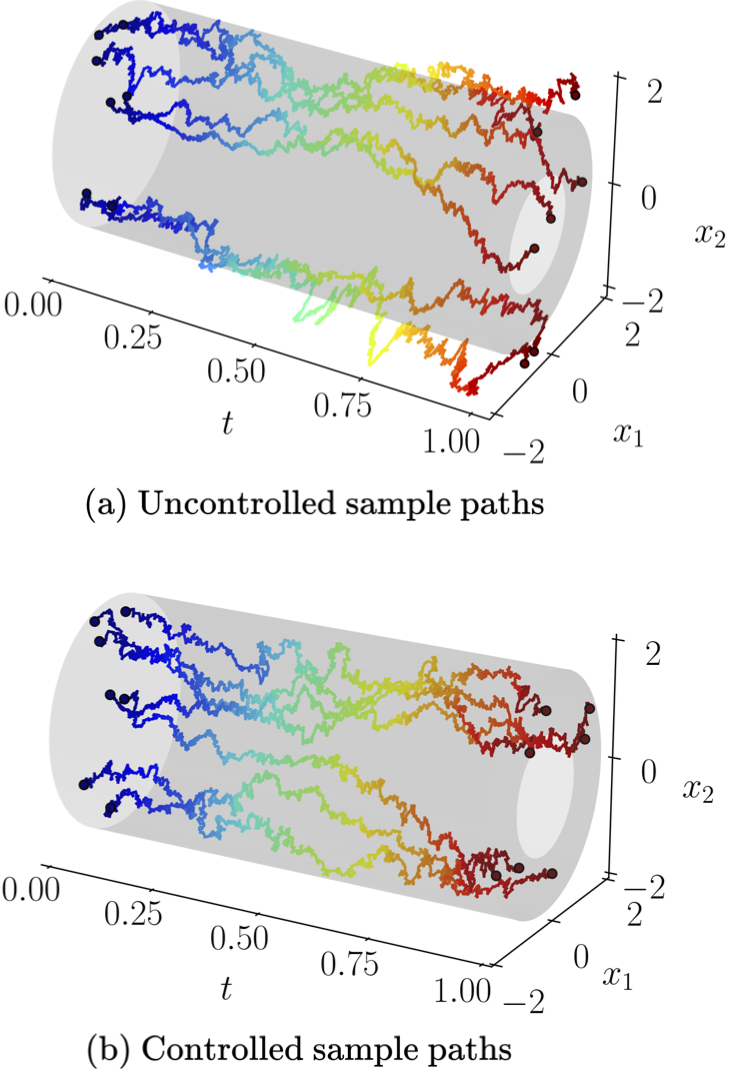}    
    \caption{(a) Uncontrolled and (b) controlled sample paths for \textbf{Example 2} with $r_1=1$, $r_2=2$, $\mathcal{I}=[0,1]$. With probability one, the controlled paths are guaranteed to be in $\mathcal{X}\equiv{\rm{Disk}}\left(\bm{0},r_2\right)$ till $T=1$, and to hit $\mathcal{X}_{T}\equiv{\rm{Annulus}}\left(r_1,r_2\right)$ at $T=1$. The color (blue to red) along each sample path indicates the progression of time.}
    \label{fig:2d_traj_annulus}
\end{figure}

\noindent\textbf{Example 3 (Effect of input matrix $\bm{G}$)}\\ 
\noindent To give falsification example, let us re-visit \textbf{Example 2} with the same problem data \eqref{ProblemDataExample2} except for single input, i.e., $\bm{G}(t,x_1,x_2)\in\mathbb{R}^{2\times 1}$. Then the computational \textbf{step 1}, and hence the score vector field $\bm{s}_{T}=\nabla_{\bm{x}}h_{T}$ remains the same as in \textbf{Example 2}. For \textbf{step 2}, we need to solve for $u\in\mathbb{R}$ from the static linear system 
$$\bm{G}(t,x_1,x_2)u(t,x_1,x_2) = \bm{s}_{T}(t,x_1,x_2).$$ 
If the given $\bm{G}\in\mathbb{R}^{2\times 1}$ is not collinear with the computed $\bm{s}_{T}$ for all $(t,x_1,x_2)$, then the controlled set invariance is falsified. Similar falsification follows for the two input case if $\bm{G}(t,x_1,x_2)\in\mathbb{R}^{2\times 2}$ is singular for some $(t,x_1,x_2)$.

\begin{figure*}[t]
    \centering      \includegraphics[width=0.96\textwidth]{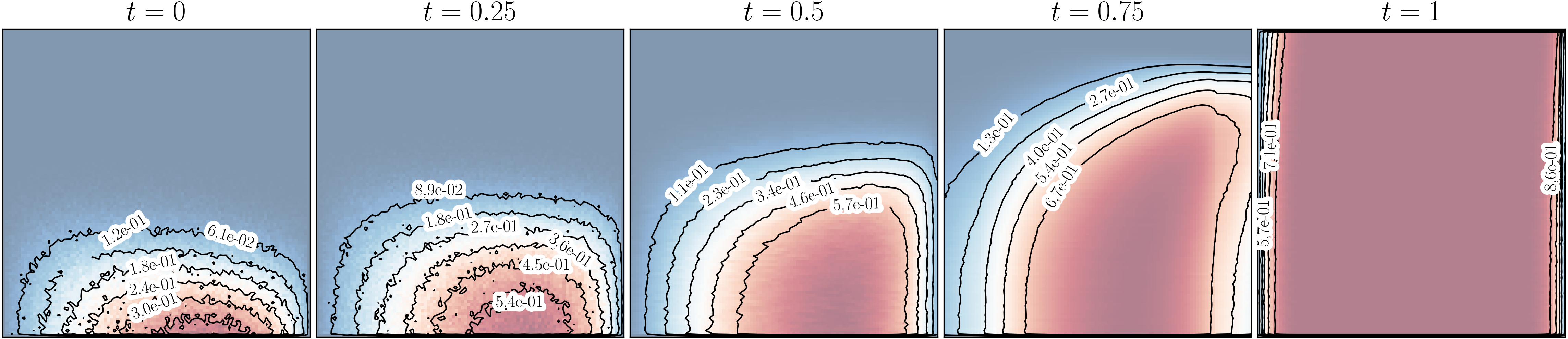}
    \caption{Contour snapshots of $h_{T}(t, \bm{x})$ for \textbf{Example 4} with $\alpha=\beta=1$, $f(x) = x^3$, $\bm{\ell} = (2,2)^{\top}$, $\mathcal{I}=[0,1]$, using the Feynman-Kac path integral computation. Each contour snapshot is plotted over the set $\mathcal{X}\equiv{\rm{Rect}}\left({\bm{0},\bm{\ell}}\right)$.}
    \label{fig:actuation_contour}
\end{figure*}
\begin{figure}[t]
    \centering
\includegraphics[width=0.95\linewidth]{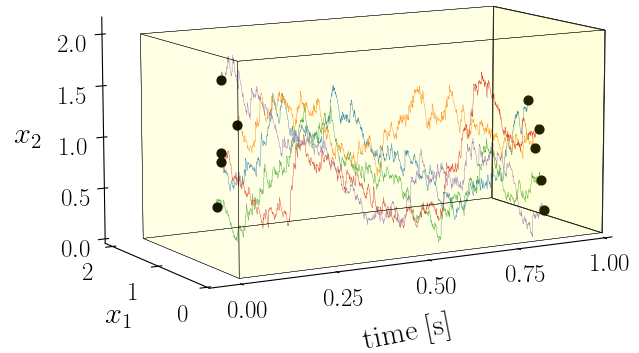}    
    \caption{The 5 controlled sample paths for \textbf{Example 4}.}
    \label{fig:2d_traj_actuator}
\end{figure}

\noindent\textbf{Example 4 (Prescribed time rectangular invariance with nonlinear drift)}\\ 
\noindent The purpose of this example is to illustrate how direct numerical computation can be performed for the proposed method with given horizon $T<\infty$. This circumvents the need for analytic handle on the solution of the associated BVP \eqref{finiteTimeDirichletPDEBVP}. 

We consider a two state single input case of \eqref{ControlledSDE} in the form
\begin{align}
    \begin{bmatrix}
         \differential x_{1t}^u\\
         \differential x_{2t}^u
    \end{bmatrix}=\left(\begin{bmatrix}
    x_{2t}^u\\
     - \alpha f\left(x_{1t}^u\right) - \beta x_{2t}^u
    \end{bmatrix}
    +\begin{bmatrix}
         0\\
         1
    \end{bmatrix} u\right) \differential t +\begin{bmatrix}
         0\\
         1
    \end{bmatrix} \differential w_t,
\label{ControlledSDENonlinearExample}    
\end{align}
with sufficiently smooth $f$, and fixed parameters $\alpha,\beta>0$. The degenerate diffusion \eqref{ControlledSDENonlinearExample} violates \textbf{Assumption 2} but we will see that the proposed computational framework still applies thanks to the same input and noise channels in \eqref{ControlledSDENonlinearExample}.

One may interpret \eqref{ControlledSDENonlinearExample} as the equations of motion for a controlled spring-mass-damper with force input having noisy actuation (thus same channel for input and noise). The $f(\cdot)$ would then model the spring nonlinearity. Alternatively, \eqref{ControlledSDENonlinearExample} can be seen as a controlled RLC circuit with linear inductor and Nyquist-Johnson resistor \cite{brockett1979stochastic}. The $f(\cdot)$ would then model the capacitor nonlinearity, i.e., capacitor voltage as nonlinear function of charge \cite[Example 3]{liberzon2000nonlinear}.

For $\bm{\ell}\in\mathbb{R}^{n}_{>0}$, we consider the prescribed time controlled set invariance for \eqref{ControlledSDENonlinearExample} with problem data
\begin{align}
&\left(\bm{f},\bm{G},\bm{\sigma},\mathcal{X},\mathcal{I}\right)\nonumber\\
=&\left(\!\begin{bmatrix}
    x_{2t}^u\\
     - \alpha f\left(x_{1t}^u\right) - \beta x_{2t}^u
    \end{bmatrix}\!\!,\begin{bmatrix}0\\1\end{bmatrix}\!\!,\begin{bmatrix}0\\1\end{bmatrix}\!\!,{\mathrm{Rect}}(\mathbf{0},\bm{\ell}),[0,T]\right).
\label{ProblemDataExample3}
\end{align}
In this case, the proposed computational \textbf{step 1} involves solving the Dirichlet BVP
\begin{subequations}
\begin{align}
&\dfrac{\partial h_T}{\partial t} + x_2 \frac{\partial h_T}{\partial x_1} - \left(\alpha f(x_1) + \beta x_2\right)\frac{\partial h_T}{\partial x_2} + \frac{1}{2}\frac{\partial^2 h_{T}}{\partial x_2^2}=0 \nonumber\\
&\qquad\qquad\qquad\qquad\forall\left(t,\bm{x}\right)\in[0,T]\times{\rm{Rect}}\left({\bm{0},\bm{\ell}}\right),\label{Example3finiteTimePDE}\\
& h_T(T,\bm{x}) = \begin{cases}1& \text{if} \quad\bm{0}<\bm{x}<\bm{\ell},\\
    0 & \text{otherwise}, \end{cases}
\label{Example3InitialCondition}\\
&h_T(t,\bm{x}) = 0 \quad\forall\:(t,\bm{x})\in[0,T)\times\partial{\rm{Rect}}\left({\bm{0},\bm{\ell}}\right).\label{Example3finiteTimeBC}
\end{align}
\label{Example3finiteTimeDirichletPDEBVP}
\end{subequations}
As for \textbf{step 2}, \eqref{BackoutToControl} and \eqref{defScoreFiniteTime} now specialize to
\begin{align}
\begin{bmatrix}
0\\
1
\end{bmatrix}u(t,x_1,x_2) = \bm{s}_{T}(t,x_1,x_2) = \begin{bmatrix}
0\\
\frac{\partial}{\partial x_2} \log h_{T}(t,x_1,x_2)
\end{bmatrix},
\label{ControlDegenerateDiffusion}
\end{align}
where $h_{T}$ solves \eqref{Example3finiteTimeDirichletPDEBVP}. The desired control is $u =\frac{\partial}{\partial x_2} \log h_{T}$.

Unlike previous examples, semi-analytical handle on the solution of \eqref{Example3finiteTimeDirichletPDEBVP} is not available in general. In such cases, we propose to compute $h_{T}$ via the Feynman-Kac path integral \cite[Ch. 8.2]{oksendal2013stochastic}, \cite{yong1997relations}, \cite[Ch. 3.3]{yong2012stochastic} where the idea is to numerically approximate $h_{T}$ as conditional expectation: 
\begin{align}
    h_T(t, \tilde{\bm{x}}) = \mathbb{E} \left[ \bm{1}_{\{\bm{x}(T)\in \mathcal{X}\}} \mid \bm{x}(t) = \tilde{\bm{x}} \right ] \quad\forall t \in [0,T),
    \label{ConditionalExpectation}
\end{align}
wherein $\bm{1}_{\{\bm{x}(T)\in \mathcal{X}\}}$ denotes the indicator function of the event $\{\bm{x}(T)\in \mathcal{X}\}$, and $\bm{x}(t)$ are the sample paths of the associated uncontrolled It\^{o} SDE \eqref{UnconditionedSDE}. This allows for approximating the (deterministic) solution of a backward-in-time PDE BVP via forward-in-time sample path simulation. For applications of Feynman-Kac path integral computation in learning and control, see e.g., \cite{pra1990markov,theodorou2010learning,nodozi2022schrodinger}.

Fig. \ref{fig:actuation_contour} shows the contour snapshots for $h_{T}$ solving \eqref{Example3finiteTimeDirichletPDEBVP} for $\alpha=\beta=1$, $f(x) = x^3$, $\bm{\ell} = (2,2)^{\top}$, $\mathcal{I}=[0,1]$. For this computation, we implemented the Feynman-Kac path integrals using the Euler-Maruyama scheme with time step $10^{-3}$ and with $10^{4}$ spatially uniform gridpoints over $[0,2]^2$. We used $1000$ sample paths for each of the gridpoints to empirically approximate the path integral.

Fig. \ref{fig:2d_traj_actuator} shows $5$ controlled sample paths obtained from the closed loop simulation via Euler-Maruyama with the same time step as before, and with the $h_{T}$ obtained from the path integral computation discussed above.

\subsection{The case $\mathcal{I}=[0,\infty)$}\label{subsec:numerical_InfiniteHorizon}
We next exemplify the proposed controlled set invariance computation for infinite horizon. In particular, \textbf{Example 5} is the infinite horizon version of \textbf{Example 1}. \textbf{Example 6} is about a.s. controlled invariance in a disk. \textbf{Example 7} is more general in that no analytic handle on $\psi_0$ is known, and is numerically computed via inverse power iteration \cite[Ch. 4.1.3]{saad2011numerical}.


\noindent\textbf{Example 5 (Infinite time hyper-rectangle invariance with zero 
drift)}\\
\noindent For $\bm{\ell}\in\mathbb{R}^{n}_{>0}$, consider the controlled set invariance of a given hyper-rectangle ${\rm{Rect}}\left({\bm{0},\bm{\ell}}\right)$ for all $t\geq 0$ with prior dynamics being the standard Brownian motion in $\mathbb{R}^{n}$.
The problem data are
\begin{align}
&\left(\bm{f},\bm{G},\bm{\sigma},\mathcal{X},\mathcal{I}\right)=\left(\bm{0},\bm{I}_{n},\bm{I}_{n},{\rm{Rect}}\left({\bm{0},\bm{\ell}}\right),[0,\infty)\right).  
\label{ProblemDataExample5}
\end{align}

Here, the Dirichlet eigenvalue problem \eqref{infiniteTimeDirichletPDEBVP} specializes to
\begin{subequations}
\begin{align}
-\frac{1}{2}\Delta_{\bm{x}}\psi_0(\bm{x}) &= \lambda_0 \psi_0(\bm{x}) \quad\forall\:\bm{x}\in{\rm{Rect}}\left({\bm{0},\bm{\ell}}\right),\label{infiniteTimeLaplacianPDE}\\
\psi_0(\bm{x})&=0 \quad\forall \bm{x}\in \partial{\rm{Rect}}\left({\bm{0},\bm{\ell}}\right).\label{infiniteTimeLaplacianBC}
\end{align}
\label{infiniteTimeLaplacianBVP}
\end{subequations}
The principal eigenvalue-eigenfunction pair $\left(\lambda_0, \psi_0\right)$ for \eqref{infiniteTimeLaplacianBVP} is \cite[Sec. 3.1]{grebenkov2013geometrical} 
\begin{align}
\left(\lambda_0, \psi_0\right) = \left(\frac{\pi^2}{2} \sum\limits_{i=1}^n \frac{1}{\ell_i^2}, \; \prod_{i=1}^n \sin\left( \frac{\pi x_i}{\ell_i}\right)\right).
    \label{GroundEigPairHalfLaplacian}
\end{align}
For $n=1$, $\ell=1$, Fig. \ref{fig:1d_infinite}(a) shows the $\psi_0$, and Fig. \ref{fig:1d_infinite}(b) shows 5 controlled sample paths. Specifically, in Fig. \ref{fig:1d_infinite}(a), the analytical $\psi_0$ from \eqref{GroundEigPairHalfLaplacian} is compared with the numerical estimate for the same obtained from inverse power iteration. The latter was applied to the discrete version of \eqref{infiniteTimeLaplacianBVP} in the domain $[0,1]$ with spatial discretization length $10^{-3}$. The sample paths in Fig. \ref{fig:1d_infinite}(b) were obtained by the closed loop Euler-Maruyama simulation with the same time discretization as in earlier examples. Similar plots for $n=2$, $\ell=(2,2)^{\top}$, are shown in
Fig. \ref{fig:2d_infinite}.
\begin{figure}[t]
    \centering
    \subfigure[$\psi_{0}(x)$ ]{
        \includegraphics[width=0.23\textwidth]{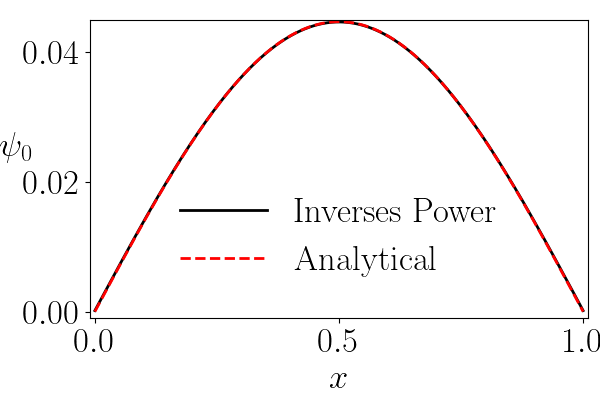}
    }
    \hspace{-10pt}
    \subfigure[5 controlled sample paths]{
        \includegraphics[width=0.23\textwidth]{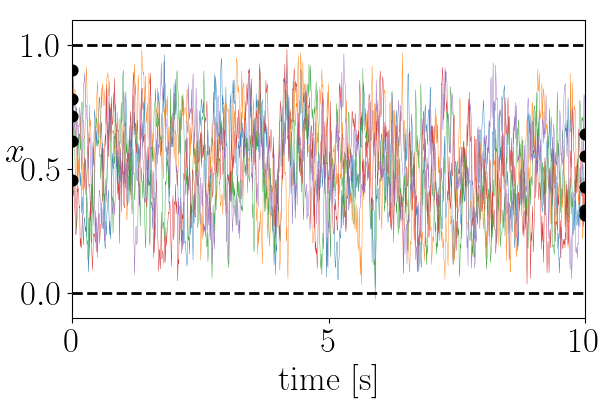}
    }
    \caption{The principal eigenfunction $\psi_0$ and 5 controlled sample paths for \textbf{Example 5} with $n=1$, $\ell=1$.}
    \label{fig:1d_infinite}
\end{figure}
\begin{figure}[htbp]
    \centering
    \subfigure[Contour plot of $\psi_{0}(\bm{x})$]{
        \includegraphics[width=0.21\textwidth]{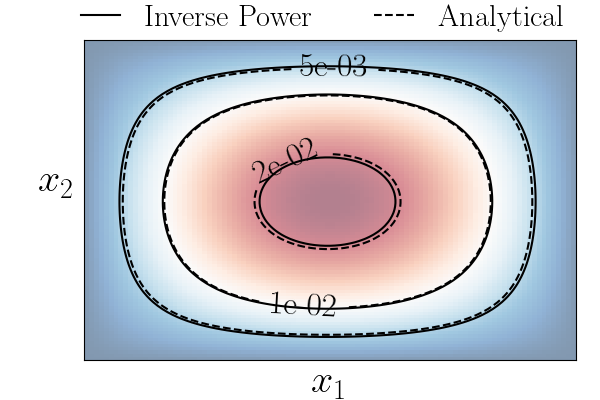}
    }
    \hspace{-20pt}
    \subfigure[5 controlled sample paths ]{
        \includegraphics[width=0.255\textwidth]{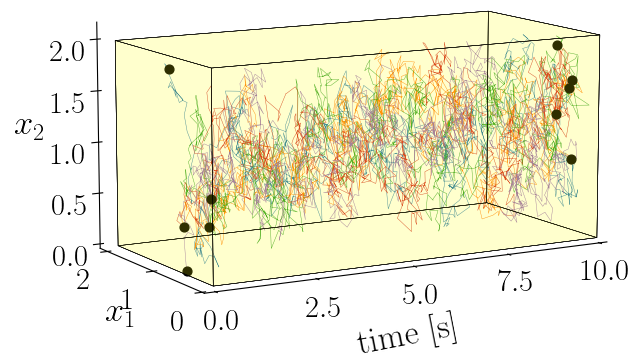}
    }
    \caption{The principal eigenfunction $\psi_0$ and 5 controlled sample paths for \textbf{Example 5} with $n=2, \ell=(2,2)^{\top}$. The $\psi_0$ contours are plotted over the set $\mathcal{X}\equiv{\rm{Rect}}\left({\bm{0},\bm{\ell}}\right)$.}
    \label{fig:2d_infinite}
\end{figure}

\noindent\textbf{Example 6 (Infinite time disk invariance with zero drift)}\\ 
\noindent
For given $r_1>0$, consider the a.s. controlled set invariance of ${\rm{Disk}}\left(\bm{0},r_1\right)$ for all $t\geq 0$ with prior dynamics being the standard Brownian motion in $\mathbb{R}^{2}$. In this case, the problem data are $\left(\bm{f},\bm{G},\bm{\sigma},\mathcal{X},\mathcal{I}\right)=\left(\bm{0},\bm{I}_{2},\bm{I}_{2},{\rm{Disk}}\left(\bm{0},r_1\right),[0,\infty)\right)$, and we specialize \eqref{infiniteTimeDirichletPDEBVP} in polar coordinates as 
\begin{subequations}
\begin{align}
&-\frac{1}{2}\left(\frac{\partial^2 \psi_0}{\partial r^2}+\frac{1}{r}\dfrac{\partial \psi_0}{\partial r}+\frac{1}{r^2}\frac{\partial^2 \psi_0}{\partial 
 \theta^2}\right) = \lambda_0 \psi_0(r,\theta)\nonumber\\
 &\hspace*{1.5in}\forall(r,\theta)\in \text{Disk}(\bm{0},r_1),\label{infiniteTimePDEdisk}\\
&\psi_0(r,\theta)=0, \quad\forall (r,\theta)\in \partial\text{Disk}(\bm{0},r_1).\label{infiniteTimeBCdisk}
\end{align}
\label{infiniteTimedisk}
\end{subequations}
\begin{figure}[t]
    \centering
\includegraphics[width=0.6\linewidth]{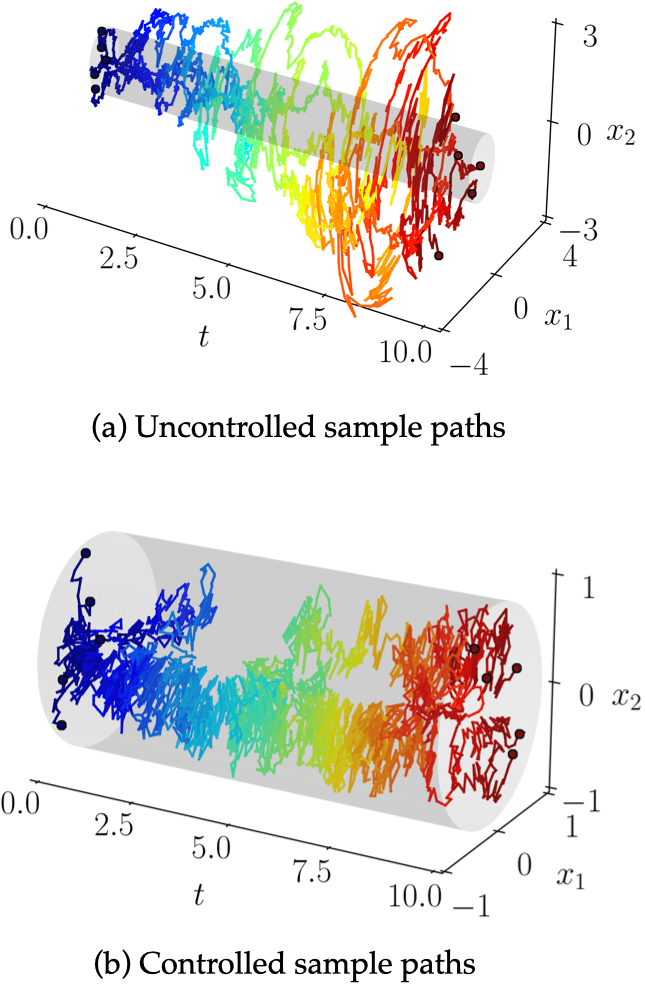}    
    \caption{(a) Uncontrolled and (b) controlled sample paths for \textbf{Example 6} with $r_1=1$. With probability one, the controlled paths are guaranteed to be in $\mathcal{X}\equiv{\rm{Disk}}\left(\bm{0},r_1\right)$. The color (blue to red) along each sample path indicates the progression of time.}
    \label{fig:2d_traj_disk}
\end{figure}
The principal eigenvalue-eigenfunction pair for \eqref{infiniteTimedisk} is \cite[Prop. 1.2.14]{henrot2006extremum}
\begin{align}
\left(\lambda_0,\psi_0\right)=\left(\left(\frac{z_1^{0}}{r_1}\right)^{2},\:\frac{1}{\sqrt{\pi}r_1 |J_{0}^{\prime}(z_1^0)|}J_0\left(\frac{z_1^{0} r}{r_1}\right)\right),
\label{GroundEigPairDisk}   
\end{align}
wherein as before, $J_0(\cdot)$ is the Bessel function of the first kind with order zero, $J_0^{\prime}(\cdot)$ denotes its derivative, and  $z_1^0\approx 2.4048$ is the first zero of $J_0(\cdot)$.

For $r_1=1$, Fig. \ref{fig:2d_traj_disk} compares 5 uncontrolled and controlled sample paths computed using \eqref{InFiniteTimeSDE} with $\psi_0$ as in \eqref{GroundEigPairDisk}.

\noindent\textbf{Example 7 (Infinite time hyper-rectangle invariance with non-zero drift)}\\ 
\noindent
We now consider problem data $\left(\bm{f},\bm{G},\bm{\sigma},\mathcal{X},\mathcal{I}\right)=\left(\begin{bmatrix}0.01-x_2\\0\end{bmatrix},\bm{I}_{2},\bm{I}_{2},[-1,1]^2,[0,\infty)\right)$. 

Unlike the previous examples, here analytic handle on $\psi_0$ is unavailable. Fig. \ref{fig:2d_follower}(a) shows the principal eigenfunction $\psi_0$ computed using the inverse power iteration. Fig. \ref{fig:2d_follower}(b) shows 5 controlled sample paths computed with this $\psi_0$.

\begin{figure}[htbp]
    \centering
    \subfigure[Contour plot of $\psi_{0}(\bm{x})$ ]{
        \includegraphics[width=0.21\textwidth]{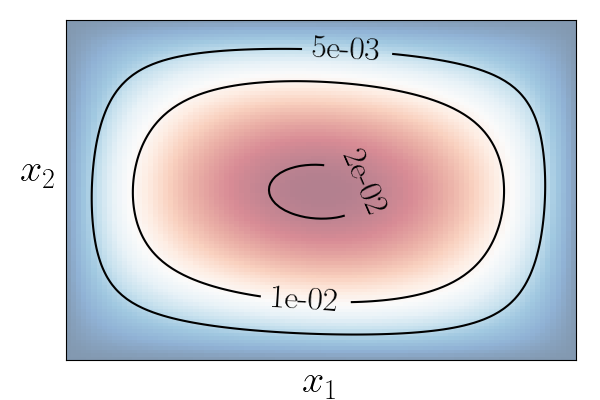}
    }
    \hspace{-10pt}
    \subfigure[Controlled sample paths ]{
        \includegraphics[width=0.255\textwidth]{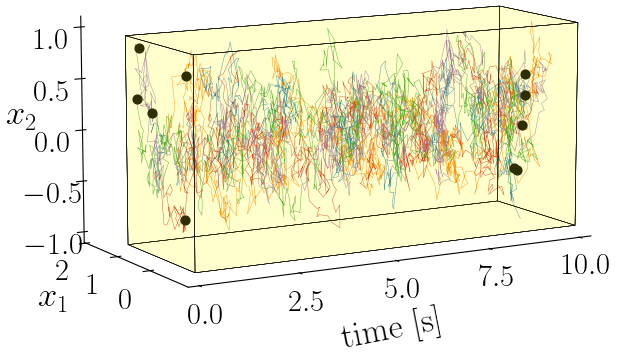}
    }
    \caption{The principal eigenfunction $\psi_0$ and 5 controlled sample paths for \textbf{Example 7}.  The $\psi_0$ contours are plotted over the set $\mathcal{X}\equiv[-1,1]^2$.}
    \label{fig:2d_follower}
\end{figure}

\noindent\textbf{Example 8 (Set invariance as non-collision)}\\ 
\noindent The purpose of this example is to highlight that for special choices of $\mathcal{X}$, we may compute the conditioned process as one that satisfies non-collision of the coordinates of the controlled sample paths. In particular, let $\mathcal{X}\equiv\mathcal{W}_{n}$ be the Weyl chamber of type ${\mathrm{A}}_{n-1}$ \cite[Ch. 15]{fulton2013representation}, and $\partial\mathcal{W}_{n}$ be its boundary, i.e.,
\begin{subequations}
\begin{align}
\mathcal{W}_{n}&:=\{\bm{x}\in\mathbb{R}^{n}\mid x_{1} < x_{2} < \hdots < x_{n}\},\\
\partial\mathcal{W}_{n} &:= \{\bm{x}\in\mathbb{R}^{n}\mid x_{1} \leq x_{2} \leq \hdots \leq x_{n},\nonumber\\
&\qquad\qquad\qquad x_{i}=x_{j}\;\text{for some}\;i\neq j\}.
\end{align}
\label{defWeylChamber}    
\end{subequations}
Guaranteeing controlled a.s. set invariance with problem data $\left(\bm{f},\bm{G},\bm{\sigma},\mathcal{X},\mathcal{I}\right)=\left(\bm{0},\bm{I}_{2},\bm{I}_{2},\mathcal{W}_{n},[t_0,\infty)\right)$ is then equivalent to condition $n$ 1D Brownian sample paths to remain non-colliding for all $t>0$. 

At first glance, this conditioning counters intuition because the 1D sample paths of the prior Brownian process collide with probability one, so the desired conditioning is w.r.t. a zero probability event. Nevertheless, it turns out to be possible in a limiting sense \cite{grabiner1999brownian}. While this computation is of slightly different flavor, we outline its rudiments as they help shed light on the connections between the finite and infinite horizons. For details, we refer the readers to \cite{grabiner1999brownian,katori2007noncolliding}.

The starting point is to invoke the Karlin-McGregor formula \cite{karlin1959coincidence} that gives a determinantal expression for the transition density $p$ for the \emph{absorbing Brownian motion in} $\mathcal{W}_{n}$, given by
\begin{align}
p(t_0,\bm{x}_0,t,\bm{x}) = \det\left[\kappa\left(t_0,x_{0i},t,x_{j}\right)\right]_{1 \leq i,j\leq n}, \, \forall\:\bm{x}_{0},\bm{x}\in\mathcal{W}_{n},
\label{KarlinMcGregor}    
\end{align}
where $$\kappa(s,x,t,y) := \frac{\exp\left(-\frac{(x-y)^{2}}{2(t-s)}\right)}{\sqrt{2\pi (t-s)}}, \quad 0\leq s < t, \quad x,y\in\mathbb{R},$$ 
is the 1D Brownian or heat kernel. 

For $p$ as in \eqref{KarlinMcGregor}, let
\begin{align}
g_{T}(\tau,\bm{x}):=\!\int_{\mathcal{W}_{n}}\!p(t_0,\bm{x}_0,\tau,\bm{x})\differential\bm{x}_{0}\quad\forall\tau>0, \: \bm{x}\in\mathcal{W}_{n},
\label{defgT}    
\end{align}
and $h_{T}(t,\cdot):=g_{T}(T-t,\cdot)$ for all $t\in[t_0,T]$ with fixed $T<\infty$. Then \eqref{ConditionedTransitionDensity} gives the transition probability density $\widetilde{p}_{T}$ for a Brownian path from $\bm{x}_{0}\in\mathcal{W}_{n}$ at $t_0$, to $\bm{x}\in\mathcal{W}_{n}$ at $t$, conditioned to stay in $\mathcal{W}_{n}$ without collision for all $t\in[0,T]$. 

Even though this finite horizon transition density $\widetilde{p}_{T}$ is time inhomogeneous, passing to the limit $T\uparrow\infty$ produces a time homogeneous transition probability density. Specifically, substituting \eqref{KarlinMcGregor} in \eqref{defgT}, followed by some computation \cite[Sec. 3]{katori2007noncolliding} yields an asymptotic estimate 
\begin{align}
g_{T}\left(\tau,\bm{x}\right) = c_{n} \tau^{-\frac{n(n-1)}{4}} {\mathrm{Vander}}(\bm{x})\bigg\{\!1 + \mathcal{O}\left(\frac{\|\bm{x}\|_2}{\sqrt{\tau}}\right)\!\!\bigg\}
\label{gTasymptotic}    
\end{align}
as $\frac{\|\bm{x}\|_2}{\sqrt{\tau}}\downarrow 0$, wherein the Vandermonde determinant
\begin{align}
{\mathrm{Vander}}(\bm{x}):= \det\left[x_{i}^{j-1}\right]_{1\leq i,j\leq n} = \prod_{1\leq i < j \leq n}\left(x_{i}-x_{j}\right),
\label{VandermondeDet}    
\end{align}
the constant $c_{n}:=\pi^{-\frac{n}{2}}\prod_{j=1}^{n}\Gamma(j/2)/\Gamma(j)$, and $\Gamma(\cdot)$ denotes the Gamma function. 

Thus,
\begin{align}
\widetilde{p}_{\infty}\left(t_{0},\bm{x}_{0},t,\bm{x}\right)&:=\lim_{T\uparrow\infty}\widetilde{p}_{T}\left(t_{0},\bm{x}_{0},t,\bm{x}\right) \nonumber\\
&=  \bigg\{\lim_{T\uparrow\infty}\dfrac{h_{T}(t,\bm{x})}{h_{T}(t_0,\bm{x}_0)}\bigg\}p\left(t_{0},\bm{x}_{0},t,\bm{x}\right)\nonumber\\
&=\bigg\{\lim_{T\uparrow\infty}\dfrac{g_{T}(T-t,\bm{x})}{g_{T}(T-t_0,\bm{x}_0)}\bigg\}p\left(t_{0},\bm{x}_{0},t,\bm{x}\right)\nonumber\\
&=\frac{{\mathrm{Vander}}(\bm{x})}{{\mathrm{Vander}}(\bm{x}_{0})} p\left(t_{0},\bm{x}_{0},t,\bm{x}\right),
\label{InfiniteHorizonTransitionDensityNoCollisionBrownian}
\end{align}
which says that noncolliding Brownian motion is the
Doob's $h$-transform of the absorbing Brownian motion in $\mathcal{W}_{n}$, where ${\mathrm{Vander}}(\bm{x})$ plays the role of the (spatial) harmonic function. Indeed, direct computation verifies $\Delta_{\bm{x}}{\mathrm{Vander}}(\bm{x}) = 0$.

\begin{figure}[htbp]
    \centering
    \subfigure[1D Brownian paths $\{x_{i}(t)\}_{i=1}^{3}$ conditioned on $x_1(t) < x_2(t) < x_3(t)$ $\forall t\in[0,\infty)$. Initial locations in blue, terminal (here $t=10$) locations in red.]{       \includegraphics[width=0.99\linewidth]{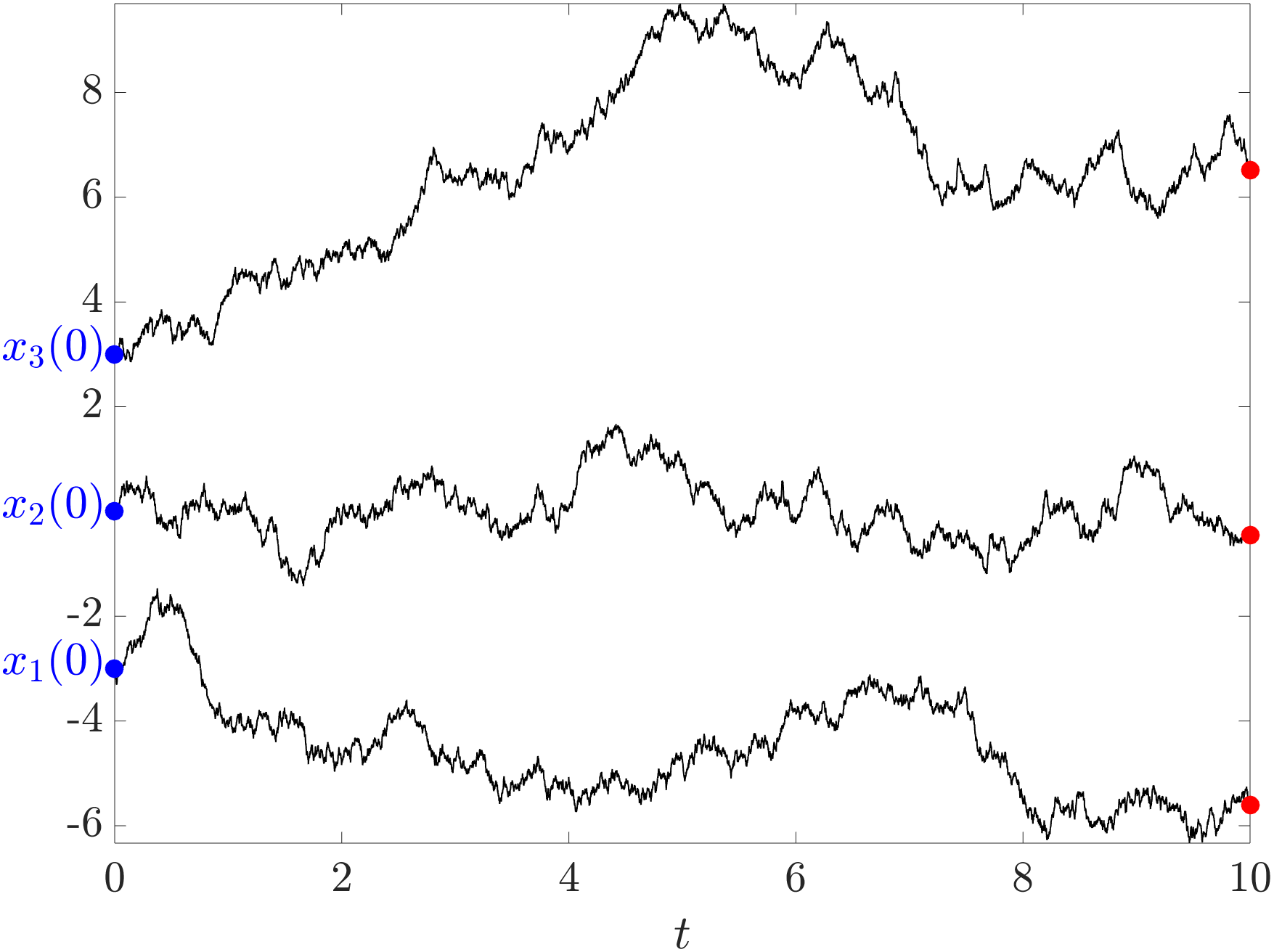}
    }
    \subfigure[The sample path $(x_1(t),x_2(t),x_3(t))^{\top}$ from part (a) shown within the Weyl chamber $\mathcal{W}_{3}$ (shaded region).]{        \includegraphics[width=0.99\linewidth]{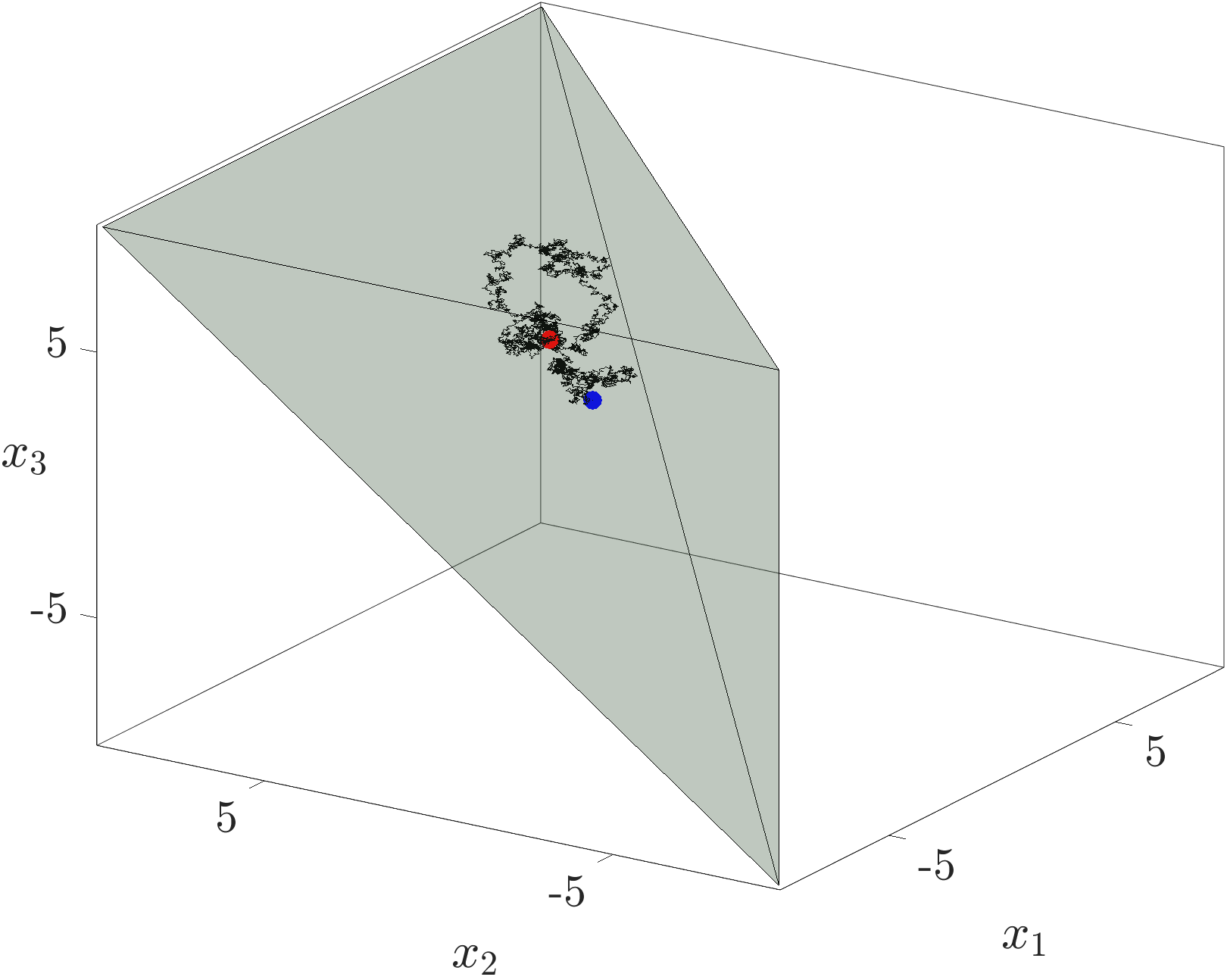}
    }
    \caption{The a.s. infinite horizon controlled invariance of the Weyl chamber $\mathcal{W}_{3}$ for \textbf{Example 8}.}
    \label{fig:DysonExample}
    \vspace*{-0.1in}
\end{figure}

Consequently, the time-invariant state feedback 
\begin{align}
\bm{u}(\bm{x}) = \nabla_{\bm{x}}\log{\mathrm{Vander}}(\bm{x}) = \begin{pmatrix}
\sum_{j\neq 1}\frac{1}{x_{1}-x_{j}}\\
\vdots\\
\sum_{j\neq n}\frac{1}{x_{n}-x_{j}}
\end{pmatrix}
\label{GradOfVandermonde}    
\end{align}
achieves the desired non-collision for all $t\in[t_0,\infty)$. The conditioned SDE \eqref{InFiniteTimeSDE} takes the form of the Dyson Brownian motion having a nonlinear repulsive drift\footnote{Alternatively, one can directly verify that \eqref{InfiniteHorizonTransitionDensityNoCollisionBrownian} solves the PDE initial value problem \eqref{TransitionPDFIVP} with $L^{\dagger}\widetilde{p}_{\infty} = \nabla_{\bm{x}}\cdot\left(\widetilde{p}_{\infty}\nabla_{\bm{x}}\log{\mathrm{Vander}}(\bm{x})\right) -\frac{1}{2}\Delta_{\bm{x}}\widetilde{p}_{\infty}$, and using \eqref{defAdjoint}, infer the corresponding sample path dynamics \eqref{DysonBMSDE}.}:
\begin{align}
        \differential\widetilde{x}_{i}(t) &= u_{i}\left(\widetilde{\bm{x}}(t)\right)\differential t + \differential w_{i}(t)\nonumber\\
        &= \sum_{j\neq i}\dfrac{\differential t}{\widetilde{x}_{i}(t)-\widetilde{x}_{j}(t)} + \differential w_{i}(t), \quad\forall i=1,\hdots, n.
\label{DysonBMSDE}    
\end{align}

Remarkably, the SDE \eqref{DysonBMSDE} was introduced by Dyson \cite{dyson1962brownian} for eigenvalue dynamics of Hermitian matrix-valued diffusion in the Gaussian unitary ensemble (GUE). In the random matrix theory literature \cite[Ch. 3.1]{tao2012topics}, \eqref{InfiniteHorizonTransitionDensityNoCollisionBrownian} is referred to as the Johansson formula \cite{johansson2001universality}, and signifies the joint density of eigenvalues for Hermitian random matrices in GUE. For set invariance in other types Weyl chambers and their connections with various random matrix diffusions, see e.g., \cite{grabiner1999brownian,konig2001eigenvalues,katori2004symmetry}.

Fig. \ref{fig:DysonExample} illustrates the a.s. controlled invariance of $\mathcal{W}_{3}$ using the Euler-Maruyama simulation with \eqref{DysonBMSDE} and time step-size $10^{-3}$.


\section{Conclusion}\label{sec:conclusion}
We proposed a score-based computational approach to certify or falsify the a.s. invariance for a stochastic control system modeled by a controlled diffusion. The method builds upon the Doob's $h$-transform, and is flexible w.r.t. the time horizon of interest. The computation involves solving certain Dirichlet boundary value problems that are amenable in general by the Feynman-Kac path integral in the given finite horizon case, and by the inverse power iteration in the infinite horizon case. Theoretical and numerical results are presented to explain the ideas. The work opens up several new directions. For instance, when a.s. invariance fails, one may design a controller to maximize the first exit time. It may also be of interest to condition the prior diffusion w.r.t. other desired space-time events beyond set invariance (e.g., collision avoidance between sample paths), and thereby deriving controllers guaranteeing the same. These will be pursued in our future works. 


\section*{Acknowledgment}
The authors acknowledge helpful technical discussions with Venkatraman Renganathan and Paul Griffioen during the preliminary stage of this work.


\appendix
\subsection{Proof for Theorem \ref{Thm:GivenFiniteTimeInvariance}}\label{App:ProofThm:GivenFiniteTimeInvariance}
We prove in the following sequence: \eqref{finiteTimeDirichletPDEBVP} $\rightarrow$ \eqref{FiniteTimeSDE} $\rightarrow$ \eqref{NewGenAsComposition} $\rightarrow$ \eqref{NewGenAsGamma}.

\noindent\textbf{Proof for \eqref{finiteTimeDirichletPDEBVP}.} Writing probability as expected value of the indicator function, we express \eqref{defhT} as 
\begin{align}
h_{T}(t,\bm{x})=\mathbb{E}_{\bm{x}}\left[\mathbf{1}_{\tau_{\mathcal{X}}>T}\right].
\label{hTexpectation}
\end{align}
From \eqref{defExpectation}-\eqref{BackwardPDE}, we already see that \eqref{hTexpectation} solves the backward Kolmogorov PDE \eqref{finiteTimePDE}. By the Feynman-Kac formula \cite[Ch. 5, Thm. 7.6]{karatzas2014brownian}, \eqref{hTexpectation} precisely corresponds to the solution of the Dirichlet BVP \eqref{finiteTimeDirichletPDEBVP}. The existence-uniqueness of solution for \eqref{finiteTimeDirichletPDEBVP} is guaranteed \cite[Ch. 1]{friedman2008partial}, \cite[p. 366-368]{karatzas2014brownian} under the stated assumptions on $\bm{f},\bm{\sigma}$.

\noindent\textbf{Proof for \eqref{FiniteTimeSDE}.} To derive the sample path dynamics for $\widetilde{\bm{x}}_{t}$, we notice that the absolute continuity $\widetilde{\mathbb{P}}_{\bm{x}}\vert_{\mathcal{F}_{t}}\ll\mathbb{P}_{\bm{x}}\vert_{\mathcal{F}_{t}}$ implies that $\bm{x}_{t}$ and $\widetilde{\bm{x}}_{t}$ have the same diffusion coefficients. The new drift coefficient $\widetilde{\bm{f}}$ for $\widetilde{\bm{x}}_{t}$, by Girsanov's theorem \cite[Thm. 8.6.6]{oksendal2013stochastic}, \cite[Ch. 3.5]{karatzas2014brownian}, satisfies\footnote{Recall that $\bm{\Sigma}$ is invertible per Assumption \textbf{A2}.}
\begin{align}
&\dfrac{\differential\widetilde{\mathbb{P}}_{\bm{x}}}{\differential\mathbb{P}_{\bm{x}}}\bigg\vert_{\mathcal{F}_{t}} = \exp\left(\displaystyle\int_{0}^{t}\langle\bm{\Sigma}^{-1}(\widetilde{\bm{f}}-\bm{f}),\bm{\sigma}\:\differential\bm{w}_{s}\rangle\right.\nonumber\\
&\qquad\qquad\qquad\quad\left.-\dfrac{1}{2}\int_{0}^{t}\langle\widetilde{\bm{f}}-\bm{f},\bm{\Sigma}^{-1}(\widetilde{\bm{f}}-\bm{f})\rangle\differential s\right).
\label{Girsanov}
\end{align}
Substituting \eqref{RadonNikodym} for the LHS of \eqref{Girsanov}, and then taking $\log$ to both sides, we get the stochastic integral representation
\begin{align}
&\log h_{T}(\bm{x}_{t\wedge\tau_{\mathcal{X}}}) - \log h_{T}(t,\bm{x}) = \displaystyle\int_{0}^{t}\langle\bm{\Sigma}^{-1}(\widetilde{\bm{f}}-\bm{f}),\bm{\sigma}\:\differential\bm{w}_{s}\rangle\nonumber\\
&\qquad\qquad\qquad-\dfrac{1}{2}\int_{0}^{t}\langle\widetilde{\bm{f}}-\bm{f},\bm{\Sigma}^{-1}(\widetilde{\bm{f}}-\bm{f})\rangle\differential s.
\label{loghTasStochasticIntegral}
\end{align}
Applying differential operator to both sides of \eqref{loghTasStochasticIntegral}, using It\^{o}'s lemma for $\log h_{T}$, and recalling that $\hess_{\bm{x}}\log h_{T}=\frac{1}{h_{T}}\hess_{\bm{x}}h_{T} - \left(\nabla_{\bm{x}}\log h_{T}\right)\left(\nabla_{\bm{x}}\log h_{T}\right)^{\top}$, we obtain
\begin{align}
&\left(\dfrac{1}{h_{T}}L_{0}[h_{T}] -\frac{1}{2}\langle\bm{\Sigma}\nabla_{\bm{x}}\log h_{T},\nabla_{\bm{x}}\log h_{T}\rangle\right)\differential t \nonumber\\
&+ \langle\nabla_{\bm{x}}\log h_{T},\bm{\sigma}\:\differential\bm{w}_t\rangle = -\dfrac{1}{2}\langle\widetilde{\bm{f}}-\bm{f},\bm{\Sigma}^{-1}(\widetilde{\bm{f}}-\bm{f})\rangle\differential t \nonumber\\
&\qquad\qquad\qquad\qquad\qquad\;+ \langle\bm{\Sigma}^{-1}(\widetilde{\bm{f}}-\bm{f}),\bm{\sigma}\:\differential\bm{w}_{t}\rangle,
\label{AppliedItoLemma}
\end{align}
where $L_0 = \frac{\partial}{\partial t} + L$. From \eqref{finiteTimePDE}, $L_{0}[h_{T}] = 0$ and the LHS of \eqref{AppliedItoLemma} simplifies. Comparing both sides of \eqref{AppliedItoLemma}, we then conclude: $\widetilde{\bm{f}} - \bm{f} = \bm{\Sigma}\nabla_{\bm{x}}\log h_{T}$, i.e., the drift for the conditioned process
\begin{align}
\widetilde{\bm{f}} = \bm{f} + \bm{\Sigma}\nabla_{\bm{x}}\log h_{T}.
\label{DriftOfConditionedProcess}    
\end{align}

\noindent\textbf{Proof for \eqref{NewGenAsComposition}.} Noting that \eqref{FiniteTimeSDE} differs from \eqref{UnconditionedSDE} by an additive drift term, using \eqref{defGenerator}, the new generator $$\widetilde{\gen}[\phi] = \gen[\phi] + \langle\bm{\Sigma}\nabla_{\bm{x}}\log h_T,\nabla_{\bm{x}}\phi\rangle.$$ Therefore,
\begin{align}
\left(\dfrac{\partial}{\partial t}+\widetilde{\gen}\right)[\phi] = \dfrac{\partial\phi}{\partial t} + \gen[\phi] + \langle\bm{\Sigma}\nabla_{\bm{x}}\log h_T,\nabla_{\bm{x}}\phi\rangle.
\label{LtildeEqualsLPlusExtra}
\end{align}
Now,
\begin{align}
&\left(h_T^{-1}\circ \left(\dfrac{\partial}{\partial t}+\gen\right) \circ h_T\right) \left[\phi\right] \nonumber\\
=&h_T^{-1}\left(\dfrac{\partial(h_T\phi)}{\partial t}+\gen\left[h_T\phi\right]\right) \nonumber\\
=& h_T^{-1}\left(h_T\dfrac{\partial\phi}{\partial t} + \phi\dfrac{\partial h_T}{\partial t} + h_T\langle\bm{f},\nabla_{\bm{x}}\phi\rangle + \phi\langle\bm{f},\nabla_{\bm{x}}h_T\rangle \nonumber\right.\\
&\left.\qquad\qquad\qquad\qquad\hspace{35pt}+ \dfrac{1}{2}\langle\bm{\Sigma},\hess_{\bm{x}}\left(h_T\phi\right)\rangle\right).
\label{ExpandingComposition}
\end{align}
Since $\hess_{\bm{x}}\left(h_T\phi\right) = h_T\hess_{\bm{x}}\phi + \left(\nabla_{\bm{x}}h_T\right)\left(\nabla_{\bm{x}}\phi\right)^{\top} + \phi\hess_{\bm{x}}h_T + \left(\nabla_{\bm{x}}\phi\right)\left(\nabla_{\bm{x}}h_T\right)^{\top}$, we have 
\begin{multline*}
\dfrac{1}{2}\langle\bm{\Sigma},\hess_{\bm{x}}\left(h_T\phi\right)\rangle = \frac{1}{2}\left(h_T\langle\bm{\Sigma},\hess_{\bm{x}}\phi\rangle \right. \\\left.+2\langle\bm{\Sigma},(\nabla_{\bm{x}}h_T)(\nabla_{\bm{x}}\phi)^{\top}\rangle + \phi\langle\bm{\Sigma},\hess_{\bm{x}}h_T\rangle\right),    
\end{multline*}
which we substitute in \eqref{ExpandingComposition} to obtain
\begin{align}
&\left(h_T^{-1}\circ \left(\dfrac{\partial}{\partial t}+\gen\right) \circ h_T\right) \left[\phi\right] = \left(\dfrac{\partial\phi}{\partial t} + \gen[\phi]\right) \nonumber\\
&\quad+ h_T^{-1}\phi\underbrace{\left(\dfrac{\partial h_T}{\partial t} + \gen[h_T]\right)}_{=0\;\text{from}\;\eqref{finiteTimePDE}} + \bigg\langle\!\!\bm{\Sigma}, \left(\nabla_{\bm{x}}\log h_T\right)\left(\nabla_{\bm{x}}\phi\right)^{\!\top}\!\!\bigg\rangle.
\label{SimplifyComposition}
\end{align}
Because $\bm{\Sigma}$ is symmetric, the Hilbert-Schmidt inner product
\begin{align}
&\bigg\langle\bm{\Sigma}, \left(\nabla_{\bm{x}}\log h_T\right)\left(\nabla_{\bm{x}}\phi\right)^{\top}\bigg\rangle \nonumber\\
=& \tr\left(\bm{\Sigma}\left(\nabla_{\bm{x}}\log h_T\right)\left(\nabla_{\bm{x}}\phi\right)^{\top}\right)\nonumber\\
=& \tr\!\left(\!\left(\nabla_{\bm{x}}\phi\!\right)\!\left(\nabla_{\bm{x}}\log h_T\right)^{\!\top}\!\bm{\Sigma}\!\right)\!=\!\tr\!\left(\!\left(\nabla_{\bm{x}}\log h_T\right)^{\!\top}\!\bm{\Sigma}(\nabla_{\bm{x}}\phi)\!\right)\nonumber\\
=& \tr\!\left(\!\left(\bm{\Sigma}(\nabla_{\bm{x}}\log h_T)\right)^{\top}(\nabla_{\bm{x}}\phi)\!\right) \!=\! \langle\bm{\Sigma}\nabla_{\bm{x}}\log h_T,\nabla_{\bm{x}}\phi\rangle,
\label{SimplifyInnerProduct}
\end{align}
where the second and third equality used the invariance of trace under transposition and cyclic permutation, respectively. Combining \eqref{SimplifyComposition}-\eqref{SimplifyInnerProduct} with \eqref{LtildeEqualsLPlusExtra} yields \eqref{NewGenAsComposition}.

\noindent\textbf{Proof for \eqref{NewGenAsGamma}.} Using Lemma \ref{LemmaGammaForOurGenerator}, $\langle\bm{\Sigma}\nabla_{\bm{x}}\log h_T,\nabla_{\bm{x}}\phi\rangle=2\Gamma\left(\log h_T,\phi\right)$, which combined with \eqref{LtildeEqualsLPlusExtra} gives \eqref{NewGenAsGamma}.
\qedsymbol


\subsection{Proof of Theorem \ref{Thm:ControllerFiniteTime}}
By equating the diffusion process $\bm{x}_t^{\bm{u}}$ governed by \eqref{ControlledSDE}, to the process $\widetilde{\bm{x}}_t$ governed by \eqref{FiniteTimeSDE}, we find that the necessary and sufficient condition is the solvability of the linear system $\bm{Gu}=\bm{s}_{T}$, and the statement follows. \qedsymbol


\subsection{Proof of Theorem \ref{Thm:GivenInFiniteTimeInvariance}}\label{App:ProofThm:InfiniteTimeInvariance}
\noindent\textbf{Proof for \eqref{InFiniteTimeSDE}.} Using Theorem \ref{Thm:GivenFiniteTimeInvariance}, we are to establish that the parametric limit $T\uparrow\infty$ of the It\^{o} process \eqref{FiniteTimeSDE}, yields the conditioned process \eqref{InFiniteTimeSDE}. Thus, it suffices to consider the term $\lim_{T\uparrow\infty}\nabla_{(\cdot)}\log h_{T}(t,\cdot)$.

For any given $T<\infty$, we know from Theorem \ref{Thm:GivenFiniteTimeInvariance} that $h_T\in\mathcal{C}^{1,2}\left([0,T];\mathcal{X}\right)$. From \eqref{AsymptoticExpansion}, the sequence of functions $h_{T}(t,\bm{x})$ indexed by parameter $T>0$, converges uniformly in $\mathcal{X}$ as $T\uparrow\infty$. Assuming $\nabla_{\bm{x}}h_{T}(t,\cdot)$ converges uniformly in $\mathcal{X}$ as $T\uparrow\infty$, we write
\begin{align*}
\displaystyle\lim_{T\uparrow\infty}\nabla_{(\cdot)}\log h_{T}(t,\cdot) &= \nabla_{(\cdot)}\displaystyle\lim_{T\uparrow\infty}\log h_{T}(t,\cdot)\\
&=\nabla_{(\cdot)}\log h_{\infty}(t,\cdot)\\
&=\nabla_{(\cdot)}\log\psi_{0}\left(\cdot\right).
\end{align*}
In above, the first equality interchanges the limit and the gradient \cite[Thm. 7.17]{rudin1976principles}. The second equality uses continuity of $\log$ over positive argument and the definition \eqref{AsymptoticExpansion}. The last equality is as in \eqref{LimitOfGradLog}. Therefore, the $T\uparrow\infty$ limit of the conditioned process \eqref{FiniteTimeSDE} is the process \eqref{InFiniteTimeSDE}.

\noindent\textbf{Proof for \eqref{infiniteTimeDirichletPDEBVP}.} We show that the $T\uparrow\infty$ limit of \eqref{finiteTimeDirichletPDEBVP} is \eqref{infiniteTimeDirichletPDEBVP}. Following similar uniform convergence arguments as above (now w.r.t. $t$), the $T\uparrow\infty$ limit of \eqref{finiteTimePDE} is \begin{align}
\dfrac{\partial h_{\infty}}{\partial t} + \gen\left[h_{\infty}\right]=0 \:\forall\:\left(t,\bm{x}\right)\in[0,\infty)\times\mathcal{X}.
\label{hinfinityPDE}    
\end{align}
If $\{(\lambda_i,\psi_i)\}$ is the eigenvalue-eigenfunction pair sequence for problem \eqref{infiniteTimeDirichletPDEBVP}, then the expansion $h_{\infty}(t,\bm{x})=\sum_{i}c_{i}e^{\lambda_{i}t}\psi_{i}(\bm{x})$, $c_i\in\mathbb{R}$, and thus \eqref{infiniteTimePDE}, is consistent with \eqref{hinfinityPDE}, as verified by direct substitution of the expansion in \eqref{hinfinityPDE}. 

In the infinite horizon case, the only boundary condition is the $T\uparrow\infty$ limit of \eqref{finiteTimeBC}, which is 
\begin{align}
h_{\infty}(t,\bm{x})=0\:\forall(t,\bm{x})\in[0,\infty)\times\partial\mathcal{X}.
\label{hinfinityBC}    
\end{align}  
Combining \eqref{hinfinityBC} with \eqref{AsymptoticExpansion}, we get $e^{\lambda_0 t}\psi_0(\bm{x})+o\left(e^{\lambda_0 t}\right)=0$ $\forall(t,\bm{x})\in[0,\infty)\times\partial\mathcal{X}$, or equivalently $\psi_0(\bm{x})=0$ $\forall(t,\bm{x})\in[0,\infty)\times\partial\mathcal{X}$. The latter is indeed \eqref{infiniteTimeBC}.\qedsymbol


\subsection{Proof of Theorem \ref{Thm:ControllerInfiniteTime}}
The proof is similar to that of Theorem \ref{Thm:ControllerFiniteTime}. Equating the process $\bm{x}_t^{\bm{u}}$ governed by \eqref{ControlledSDE}, to the conditioned process $\widetilde{\bm{x}}_t$ governed by \eqref{InFiniteTimeSDE}, the necessary and sufficient condition becomes the solvability of the linear system $\bm{Gu}=\bm{s}_{\infty}$, and the statement follows. \qedsymbol


\subsection{Proof of Theorem \ref{Thm:InverseOptimalityFiniteHorizon}}\label{App:ProofInverseOptimalityThmFiniteHorizon}
Following standard dynamic programming computation, the first order conditions for optimality for \eqref{StochasticOCP} is formally given by the HJB PDE initial value problem:
\begin{subequations}
\begin{align}
    \dfrac{\partial V}{\partial t} &= - \frac{1}{2} (\nabla_{\bm{x}} V)^{\top}\bm{G} \bm{G}^{\top} \nabla_{\bm{x}} V- \langle \nabla_{\bm{x}} V, \bm{f}\rangle\nonumber\\&\quad- \frac{1}{2}\langle \bm{\Sigma}, \hess_{\bm{x}}(V) \rangle + q(\bm{x}),\label{HJBPDE}\\
    V(T,\bm{x}) &= 0,
\end{align}
\label{HJBPDEIVP}
\end{subequations}
and the associated optimal control
\begin{align}
\bm{u}^{\mathrm{opt}}=\bm{G}^{\top}\nabla_{\bm{x}}V.
\label{OptimalControl}    
\end{align}

To establish the equivalence between \eqref{HJBPDEIVP} and \eqref{finiteTimeDirichletPDEBVP}, the idea is to recognize that the value function $V(t,\bm{x})$ in \eqref{HJBPDEIVP}, and the $h_{T}(t,\bm{x})$ in \eqref{finiteTimeDirichletPDEBVP}, are precisely related by Fleming's logarithmic transform \cite{fleming2005logarithmic}
\begin{align}
V(t,\bm{x}) = \log h_{T}(t,\bm{x}),
\label{hisExpV}
\end{align}
which is well-defined since $h_{T}(t,\bm{x})>0$ $\forall (t,\bm{x})\in[0,T]\times\mathcal{X}$. To verify this, note from \eqref{hisExpV} that
\begin{align}
\left(\exp (-V)\right)\dfrac{\partial h_{T}}{\partial t} = \dfrac{\partial V}{\partial t},
\label{partilhpartialt}
\end{align}
and using \eqref{defGenerator}, we have
\begin{align}
&\gen h_{T}\nonumber\\ 
&= \left(\exp V\right)\langle\bm{f},\nabla_{\bm{x}}V\rangle + \frac{1}{2}\langle \bm{\Sigma},\nabla_{\bm{x}}\circ\nabla_{\bm{x}}\left(\exp V\right)\rangle\nonumber\\
&= \left(\exp V\right)\langle\bm{f},\nabla_{\bm{x}}V\rangle + \frac{1}{2}\langle \bm{\Sigma},\nabla_{\bm{x}}\left(\left(\exp V\right)\nabla_{\bm{x}}V\right)\rangle\nonumber\\
&= \left(\exp V\right)\bigg\{\langle\bm{f},\nabla_{\bm{x}}V\rangle + \frac{1}{2}\langle \bm{\Sigma}, \left(\nabla_{\bm{x}}V\right)\left(\nabla_{\bm{x}}V\right)^{\top} + \hess_{\bm{x}}V\rangle\bigg\}.\nonumber\\
\label{LhT}
\end{align}

So if $\bm{G}\bm{G}^{\top}=\bm{\Sigma}$ for all $(t,\bm{x})\in [0,T]\times\mathcal{X}$, then \eqref{partilhpartialt}-\eqref{LhT} allow rewriting the PDE initial value problem \eqref{HJBPDEIVP} as 
\begin{subequations}
\begin{align}
    &\dfrac{\partial h_{T}}{\partial t} + \gen h_{T} - q(\bm{x})h_{T} = 0\quad\forall \bm{x}\in\mathcal{X}\cup\partial\mathcal{X},
    \label{hTPDEgeneral}\\
    &h_{T}(T,\bm{x}) = 1\quad\forall \bm{x}\in\mathcal{X}.
    \label{hTterminal} 
\end{align}
\label{general_pde_ivp}
\end{subequations}
Notice that the terminal condition \eqref{hTterminal} is indeed \eqref{InitialCondition}.

Next, we specialize \eqref{hTPDEgeneral} for two cases. For the case $(t,\bm{x})\in [0,T]\times\mathcal{X}$, we have $q=0$ from \eqref{state_cost}, and \eqref{hTPDEgeneral} reduces to \eqref{finiteTimePDE}. For the case $(t,\bm{x})\in [0,T]\times\partial\mathcal{X}$, we resort to the Feynman-Kac representation for the solution of \eqref{hTPDEgeneral}:
\begin{align}
    h_{T}(t, \bm{x}) = \mathbb{E}\left[ \exp{\left( - \int_{t}^{T} q(\bm{x}_{\tau}^{\bm{u}}) d\tau \right)}h_{T}(T, \bm{x})\vert \bm{x}_{\tau}^{\bm{u}} = \bm{x} \right].
\label{feyman-kac}
\end{align}
Since $q(\bm{x}) = \infty$ for $\bm{x} \in \partial \mathcal{X}$, so from \eqref{feyman-kac}, we get $h_T(t, \bm{x}) = 0$ when $\bm{x} \in \partial \mathcal{X}$, thus recovering the absorbing boundary condition \eqref{finiteTimeBC}.

Finally, combining \eqref{OptimalControl} and \eqref{hisExpV}, we obtain $$\bm{u}^{\mathrm{opt}}=\bm{G}^{\top}\nabla_{\bm{x}}\log h_{T},$$
which satisfies \eqref{StaticLinearSystem} provided $\bm{GG}^{\top}=\bm{\Sigma}$ for all $(t,\bm{x})\in [0,T]\times\mathcal{X}$.
\qedsymbol


\subsection{Derivation of \eqref{hTExample2}-\eqref{eq:coefficient caculation}}\label{Derivation of heat equation in annular set}
For the readers' convenience, we give here the computational details whose bits and pieces can be found in Strum-Liouville eigenvalue problem literature.

For $t\in[0,T]$, let $\tau:=T-t$, and $g_{T}(\tau,r,\theta):=h_{T}(t,r,\theta)$. Then $g_{T}$ solves the PDE
\begin{align}
\frac{\partial g_{T}}{\partial\tau} = \frac{1}{2}\left(\frac{\partial^2 g_T}{\partial r^2}+\frac{1}{r}\dfrac{\partial g_T}{\partial r}+\frac{1}{r^2}\frac{\partial^2 g_T}{\partial 
 \theta^2}\right),
\label{gTPDE}
\end{align}
and \eqref{Example2InitialCondition} implies 
\begin{align}
g_T(0,r,\theta) = \begin{cases}1& \text{if} \quad(r,\theta)\in {\mathrm{Annulus}}(r_1,r_2),\\
    0 & \text{otherwise}.\end{cases}
\label{gTinitialcondition}    
\end{align}

Substituting the separation-of-variables ansatz $g_{T}(\tau,r,\theta) = \Gamma(\tau)R(r)\Theta(\theta)$ in \eqref{gTPDE}, we get
$$\frac{1}{r} \frac{(rR')'}{R} + \frac{1}{r^2} \frac{\Theta''}{\Theta} = 2 \frac{\Gamma'}{\Gamma}=-\lambda,$$
for some constant $\lambda>0$. This yields
\begin{subequations}
\begin{align}
\Gamma(\tau)&=\exp\left(-\frac{1}{2}\lambda\tau\right),\label{eq:T_solution}\\
-\frac{\Theta''}{\Theta} &=\mu,\label{eq:theta_solution_pde}\\
\lambda r^2 + r \frac{(rR')'}{R}&=\mu,\label{eq:R_solution_pde}
\end{align}
\end{subequations}
where $\mu$ is a real constant. 

To solve \eqref{eq:theta_solution_pde} and to determine $\mu$ along the way, we recall that $\theta$ is $2\pi$ periodic, and require the same for the function $\Theta(\cdot)$, i.e., $\Theta(\theta + 2n\pi) = \Theta(\theta)\:\forall n\in\mathbb{Z}$. To enforce this, we append \eqref{eq:theta_solution_pde} with periodic boundary conditions, giving 
\begin{subequations}
\begin{align}
&\Theta'' + \mu \Theta = 0, \quad 0 < \theta < 2\pi,\\ 
&\Theta(0) = \Theta(2\pi), \quad \Theta'(0) = \Theta'(2\pi).
\end{align}
\label{eq:theta_ODE}
\end{subequations}
The periodic Sturm-Liouville problem \eqref{eq:theta_ODE} has eigenvalue-eigenfunction pairs 
\begin{align*}
&\{\{\mu_{0},\Theta_{0}(\theta)\}, \{\mu_{n},\Theta_{1n}(\theta)\}, \{\mu_{n},\Theta_{2n}(\theta)\}\}\\
=&\{\{0, 1\},\{n^{2},\cos n\theta\}, \{n^{2},\sin n\theta\}\}, \quad n\in\mathbb{N}.    
\end{align*}

The problem \eqref{eq:R_solution_pde} with $\mu_n=n^2$ becomes a singular Sturm-Liouville problem:
\begin{subequations}
\begin{align}
    &\left[ rR'\right ]'-\frac{n^2}{r}R+\lambda r R=0,\label{eq:Sturm-Liouville}\\
    &R(r_2)=0,\; R(r) \;\text{bounded} \;\text{as} \;r\rightarrow 0,\label{eq:Sturm-LiouvilleBC}
\end{align}
\label{eq:R_ode}
\end{subequations}
which admits solutions only for $\lambda>0$. To compute these solutions, we perform change of variable $r \mapsto \tilde{r}:=\sqrt{\lambda} r$ and let $\tilde{R}(\tilde{r}):=R(r)$, thus transforming \eqref{eq:Sturm-Liouville} to Bessel's ODE \cite[Ch. 9.1]{abramowitz1972handbook} of order $n$, given by
\begin{align}\label{eq:Bessel equation}
    \tilde{r}^2 \tilde{R}^{\prime\prime} + \tilde{r} \tilde{R}^{\prime} + \left(\tilde{r}^2 - n^2\right) \tilde{R} = 0,
\end{align}
wherein $^{\prime}$ denotes derivative w.r.t. $\tilde{r}$. The general solution of \eqref{eq:Bessel equation} is $\tilde{R}\left(\tilde{r}\right) = \alpha J_{n}(\tilde{r}) + \beta Y_{n}(\tilde{r})$, where $J_{n},Y_{n}$ denote the Bessel functions of the first and second kind of order $n\in\mathbb{N}$, respectively, and the constants $\alpha,\beta\in\mathbb{R}$. Hence 
\begin{align}
R(r) = \tilde{R}\left(\tilde{r}=\sqrt{\lambda} r\right) = \alpha J_{n}\left(\sqrt{\lambda} r\right) + \beta Y_{n}\left(\sqrt{\lambda} r\right).
\label{GeneralSolutionR}
\end{align}
Since $Y_{n}(\cdot)$ is unbounded as its argument approaches zero \cite[Ch. 9.1]{abramowitz1972handbook}, so from \eqref{eq:Sturm-LiouvilleBC}, we find
\begin{align}
\beta=0, \quad \alpha J_{n}\left(\sqrt{\lambda}\:r_2\right)=0.
\label{alphabeta}    
\end{align}
Then, for \eqref{GeneralSolutionR} to admit nontrivial solution, we must have $\alpha\neq 0$, which together with \eqref{alphabeta} implies $J_{n}\left(\sqrt{\lambda}\:r_2\right)=0$. 

It is known \cite[Ch. 15]{watson1922treatise} that for each $n\in\mathbb{N}$, the zeros of $J_{n}(\cdot)$ are all simple and positive. For fixed $n\in\mathbb{N}$, denoting the $k$\textsuperscript{th} zero of $J_{n}(\cdot)$ as as $z^n_k$, where $k\in\mathbb{N}$, the eigenvalues solving $J_{n}\left(\sqrt{\lambda}\:r_2\right)=0$, and thus the Sturm-Lioville problem \eqref{eq:R_ode}, must then be $\lambda^n_k=\left(\frac{z^n_k}{r_2}\right)^2$. The corresponding eigenfunctions are 
\begin{align}
\phi^n_k(r)=J_{n}\left(\sqrt{\lambda^n_k}\:r\right) = J_n\left( \frac{z^n_k r}{r_2} \right).
\label{defEigFunction}    
\end{align}

Combining the solutions of \eqref{eq:T_solution}, \eqref{eq:theta_ODE} and \eqref{eq:R_ode}, the general solution for \eqref{gTPDE} can now be written as a linear combination\footnote{since the PDE \eqref{gTPDE} and its boundary conditions are linear and homogeneous}:
\begin{align}
    &g_T(\tau=T-t,r,\theta)=h_{T}(t,r,\theta)\nonumber\\
    &=\sum\limits_{k=1}^{\infty} c^0_{k}\phi^0_k(r)\exp\left(-\frac{1}{2}(T-t)\left(\frac{z^{0}_{k}}{r_2}\right)^2 \right)\nonumber\\
    &+\hspace{-3pt}\sum\limits_{n=1}^{\infty}
    \hspace{-1pt}\sum\limits_{k=1}^{\infty} \hspace{-3pt}\phi^n_k(r)\left(c^n_k(r)\cos n\theta\hspace{-2pt}+\hspace{-2pt}d^n_k(r)\sin n\theta\right)\nonumber\\
    &\qquad\qquad\times\exp\left(-\frac{1}{2}(T-t)\!\left(\dfrac{z^{0}_{k}}{r_2}\right)^{\!\!2}\right),
    \label{eq:general solution}
\end{align}
wherein the $\phi^n_k(\cdot)$ for all $n\in\mathbb{N}\cup\{0\}$, $k\in\mathbb{N}$, are given by \eqref{defEigFunction}, and the constants $c_{k}^{0}, c_{k}^{n}, d_{k}^{n}$ for $n,k\in\mathbb{N}$, remain to be determined from the boundary conditions.

Rewriting the terminal condition \eqref{Example2InitialCondition} as the Fourier series expansion 
\begin{align}
    h_T(T,r, \theta)\hspace{-2pt}=\hspace{-2pt}a_0(r) \hspace{-2pt}+\hspace{-2pt} \sum_{n=1}^\infty \left( a_n(r) \cos n\theta + b_n(r) \sin n\theta \right),
    \label{TerminalConditionFourierSeries}
\end{align}
we find
\begin{subequations}
\begin{align}
&a_0(r) \hspace{-2pt}=\hspace{-2pt} \frac{1}{2\pi} \int_0^{2\pi} h_T(T,r, \theta) \, d\theta\hspace{-2pt}=\hspace{-2pt}\begin{cases}
0, & \text{if } \; 0\le r< r_1, \\
1, & \text{if } \; r_1\le r\le r_2,
\end{cases}\label{azero}\\
&a_n(r)\hspace{-2pt}=\hspace{-2pt} \frac{1}{\pi} \int_0^{2\pi} h_T(T,r, \theta)\cos n\theta\:d\theta\hspace{-2pt}=0,\quad n\in\mathbb{N},\label{anzero}\\
&b_n(r)\hspace{-2pt}=\hspace{-2pt} \frac{1}{\pi} \int_0^{2\pi} h_T(T,r, \theta)\sin n\theta\:d\theta\hspace{-2pt}=0,\quad n\in\mathbb{N}.\label{bnzero}
\end{align}
\label{TerminalCoeff}
\end{subequations}
Evaluating the general solution \eqref{eq:general solution} at $t=T$ and comparing the same with \eqref{TerminalConditionFourierSeries}, we get the expansions of $a_0,a_n,b_n$ in the eigenfunction basis for the Sturm-Liouville problem \eqref{eq:R_ode}:
\begin{subequations}
\begin{align}
&a_{0}(r) = \sum\limits_{k=1}^{\infty} c^0_{k}\phi^0_k(r),\\
&a_{n}(r) = \sum\limits_{k=1}^{\infty} c^n_{k}\phi^n_k(r), \; b_{n}(r) = \sum\limits_{k=1}^{\infty} d^n_{k}\phi^n_k(r),\; n\in\mathbb{N}.
\end{align}
\label{a0anbn}
\end{subequations}

Note that for the Sturm-Liouville problem \eqref{eq:R_ode}, the eigenfunctions must be orthogonal (in weighted $L^2$ sense) over the domain ${\rm{Disk}}\left({\bm{0},r_2}\right)$ with weight $w(r)=r$. Thus, taking the weighted $L^2$ inner product\footnote{The weighted $L^2$ inner product between $\varphi(r)$ and $\psi(r)$ w.r.t. the weight function $w(r)$, is $\langle\varphi(r),\psi(r)\rangle_{w(r)}:=\int\varphi(r)\psi(r)w(r)\:\differential r$.} $\langle\cdot,\cdot\rangle_{w(r)}$ to both sides of \eqref{a0anbn} with respective eigenfunctions, yield
\begin{subequations}
\begin{align}
&c_{k}^{0} = \frac{ \int_{0}^{r_2} a_0(r) \phi^0_k(r) r\differential r }{ \int_{0}^{r_2} \left(\phi^0_k(r)\right)^2 r\differential r }\stackrel{\eqref{azero}}{=}\frac{ \int_{r_1}^{r_2} \phi^0_k(r) r\differential r }{ \int_{0}^{r_2} \left(\phi^0_k(r)\right)^2 r\differential r}\quad \forall k\in\mathbb{N},\\
&c_{k}^{n} = \frac{ \int_{0}^{r_2} a_n(r) \phi^0_k(r) r\differential r }{ \int_{0}^{r_2} \left(\phi^0_k(r)\right)^2 r\differential r } \stackrel{\eqref{anzero}}{=}0\quad \forall n,k\in\mathbb{N},\\ 
&d_{k}^{n}=\frac{ \int_{0}^{r_2} b_n(r) \phi^0_k(r) r\differential r }{ \int_{0}^{r_2} \left(\phi^0_k(r)\right)^2 r\differential r }\stackrel{\eqref{bnzero}}{=}0\quad \forall n,k\in\mathbb{N}.
\end{align}
\label{ckdk}
\end{subequations}
Substituting the coefficients $c^0_k,c^n_k,d^n_k$ from \eqref{ckdk} back in \eqref{eq:general solution}, we obtain \eqref{hTExample2}-\eqref{eq:coefficient caculation}.


\section*{References}
\bibliographystyle{IEEEtran}
\bibliography{references.bib}

\balance


\end{document}